\documentclass[10pt]{amsart}
\input epsf

\usepackage{amsfonts,amsthm,amsmath,amssymb,latexsym}
\usepackage{graphicx,color}
\usepackage[all]{xy}

\begin{document}

\author{Dragomir \v Sari\' c}

\address{Department of Mathematics, CUNY Queens College,
65-30 Kissena Blvd., Flushing, NY 11367}
\email{Dragomir.Saric@qc.cuny.edu}

\address{Mathematics PhD program, CUNY Graduate Center, 365 Fifth
Avenue, New York, NY 10016-4309}

\theoremstyle{definition}

\newtheorem{definition}{Definition}[section]
\newtheorem{remark}[definition]{Remark}
\newtheorem{example}[definition]{Example}

\newtheorem*{notation}{Notation}

\theoremstyle{plain}

 \newtheorem{proposition}[definition]{Proposition}
 \newtheorem{theorem}[definition]{Theorem}
 \newtheorem{corollary}[definition]{Corollary}
 \newtheorem{lemma}[definition]{Lemma}

\def\H{{\mathbf H}}
\def\F{{\mathcal F}}
\def\R{{\mathbf R}}
\def\Q{{\mathbf Q}}
\def\Z{{\mathbf Z}}
\def\E{{\mathcal E}}
\def\N{{\mathbf N}}
\def\X{{\mathcal X}}
\def\Y{{\mathcal Y}}
\def\C{{\mathbf C}}
\def\D{{\mathbf D}}
\def\G{{\mathcal G}}

\title[Zygmund vector fields and Hilbert transform]{Zygmund vector fields, Hilbert transform and Fourier coefficients
in shear coordinates}


\date{\today}

\maketitle

\begin{abstract}
We parametrize the space $\mathcal{Z}$ of Zygmund vector fields on
the unit circle in terms of infinitesimal shear functions on the
Farey tesselation. Then we express the Hilbert transform and the
Fourier coefficients of the Zygmund vector fields in terms of the
above parametrization by infinitesimal shear functions. Finally, we
compute the Weil-Petersson metric on the Teichm\"uller space of a
punctured surface in terms of shears.
\end{abstract}

\section{Introduction}

The hyperbolic plane $\H$ has the unit circle $S^1$ as its ideal
boundary. Vector fields on the unit circle $S^1$ are describing
infinitesimal deformations of the circle maps. We study Zygmund
vector fields which arise as tangent vectors to the space of
quasisymmetric maps of $S^1$ at the identity map $id:S^1\to S^1$
(see \cite{GS}, \cite{Riem}). The Teichm\"uller space $T(\H )$ of
the hyperbolic plane $\H$, called the {\it universal Teichm\"uller
space}, consists of all quasisymmetric maps of the unit circle $S^1$
which fix $1,i,-1\in S^1$. Therefore, the space $\mathcal{Z}$ of all
Zygmund vector fields on $S^1$ which vanish at $1,i,-1\in S^1$ is
the tangent space to the universal Teichm\"uller space $T(\H )$ at
the {\it basepoint} $id\in T(\H )$ (see \cite{GL}).

The Farey tesselation $\F$ of the hyperbolic plane $\H$ is a locally
finite ideal triangulation of $\H$ which is invariant under the
group of isometries of $\H$ generated by hyperbolic reflections in
the sides of a complementary triangle of $\F$. The group of
orientation preserving isometries of $\H$ that setwise preserve $\F$
is isomorphic to $PSL_2(\mathbf{Z})$ and it acts simply transitively
on the oriented edges of $\F$. A vector field on $S^1$ naturally
induces a real-valued function on the edges of $\F$, called the {\it
infinitesimal shear function} of $V$ (see Section 4). We parametrize
the space $\mathcal{Z}$ of Zygmund vector fields on $S^1$ that
vanish at $1,i,-1\in S^1$ in terms of the corresponding shear
functions. Then we proceed to compute the Hilbert transform of
Zygmund vector fields in terms of infinitesimal shear functions. We
apply these formulas to compute the Weil-Petersson metric on
Teichm\"uller spaces of finite punctured surfaces and the Fourier
coefficients of Zygmund vector fields. Finally, we show that the
formula for the Weil-Petersson metric in shear coordinates extend by
continuity to the Weil-Petersson metric of the boundary
Teichm\"uller space in its shear coordinates (which is a result of
Masur \cite{mas}). We give more details below.

A quasisymmetric map $h:S^1\to S^1$ induces a shear function
$s:\mathcal{F}\to\R$ which assigns to each edge $e\in\F$ a real
number
$$
s(e)=\log cr(a,b,c,d)
$$
where the endpoints of $e$ are $b,d\in S^1$, the two adjacent
complementary triangles of $\F$ with common side $e$ have third
vertices $a,c\in S^1$, and
$cr(a,b,c,d)=\frac{(b-c)(d-a)}{(b-a)(d-c)}$. A {\it fan of
geodesics} of $\F$ with {\it tip} $p\in S^1$ is the set $\F_p$ of
all geodesics of $\F$ with one endpoint $p$. Fix a correspondence of
$\F_p$ with the integers $\Z$ such that $\F_p=\{e^p_n\}_{n\in\Z}$
and $e_n^p$ is adjacent to $e_{n+1}^p$ for all $n\in\Z$. We have the
following characterization of quasisymmetric maps in terms of shear
functions.

\vskip .2 cm

\noindent {\bf Theorem 1.} \cite{Sa2} {\it A shear function
$s:\F\to\R$ induces a quasisymmetric map of the unit circle if and
only if there exists a constant $M\geq 1$ such that for each fan
$\F_p=\{ e^p_n\}_{n\in\Z}$ and for each $m,k\in Z$, $k\geq 0$, we
have
$$
\frac{1}{M}\leq e^{s(e_m)}\frac{1+e^{s(e_{m+1})}+\cdots
+e^{s(e_{m+1})+\cdots +s(e_{m+k})}}{1+e^{-s(e_{m-1})}+\cdots
+e^{-s(e_{m-1})-\cdots -s(e_{m-k})}}\leq M.
$$
}

\vskip .2 cm

A differentiable path $h_t:S^1\to S^1$, $|t|<\epsilon$, of
quasisymmetric maps with $h_0=id$ induces a Zygmund vector field $V$
on $S^1$ by the formula
$$
V(z)=\frac{d}{dt}h_t(z)|_{t=0}.
$$
If $s_t:\F\to\R$ is the path of shear functions for $h_t$, then the
derivative
$$
\frac{d}{dt}s_t(e)|_{t=0}=\dot{s}(e)
$$
for $e\in\F$ induces a {\it infinitesimal shear function}
$$
\dot{s}:\F\to\R
$$
corresponding to the Zygmund vector field $V$. Our first result
given in Theorem \ref{thm:zygmund-parametrization-shears} is a
characterization of infinitesimal shear functions which induce
Zygmund bounded vector fields on $S^1$.

\vskip .2 cm

\noindent {\bf Theorem 2.} {\it Let $\dot{s}:\F\to\R$ be a shear
function. Then $\dot{s}$ induces a Zygmund vector field $V$ on $S^1$
if and only if there exists a constant $C>0$ such that for all fans
$\F_p=\{ e^p_n\}_{n\in\Z}$ and for all $m,k\in\Z$, $k\geq 0$, we
have
\begin{equation*}
\begin{split}
\Big{|}\dot{s}(e_m)+\frac{k}{k+1}[\dot{s}(e_{m+1})+\dot{s}(e_{m-1})]
+\frac{k-1}{k+1}[\dot{s}(e_{m+2}) +\\ +\dot{s}(e_{m-2})]+  \cdots
+\frac{1}{k+1}[\dot{s}(e_{m+k})+\dot{s}(e_{m-k})]\Big{|}\leq C.
\end{split}
\end{equation*}
}

\vskip .2 cm

A vector field on $S^1$ that vanishes at $1,i,-1$ is uniquely
determined by its infinitesimal shear function. Therefore the above
theorem gives an explicit (and rather simple) parametrization of the
space $\mathcal{Z}$ (of all normalized Zygmund bounded vector
fields) which is the tangent space at the basepoint of the universal
Teichm\"uller space $T(\H )$. This should be compared to the
question of characterizing Zygmund vector fields (or Zygmund maps)
in terms of its Fourier coefficients where no explicit condition is
known (see \cite{Kat}). The key idea in the proof of the above
theorem is to ``decompose'' a vector field $V$ into an infinite sum
of the vector fields $V_p$ over all fans of $\F_p$ of $\F$. The
infinitesimal shear functions of $V_p$ are zero on all geodesics of
$\F$ except on the geodesics of $\F_p$ where they are equal to (the
one-half of) the infinitesimal shear function $\dot{s}$. We obtain
(see Corollary \ref{cor:zyg_series_estimate})
$$
V(z)=\sum_{p\in\F^0}V_p(z)
$$
where $\F^0$ is the set of vertices of $\F$; the convergence is
absolute and uniform; each partial sum of $V_p$'s and the sum $V$ is
Zygmund bounded with a Zygmund constant continuously depending on
$C$. The decomposition of $V$ into vector fields $V_p$ corresponding
to the fans of $\F$ is well-suited for the Zygmund vector fields,
while writing the sum over all geodesics of $\F$ does not even
produce a pointwise convergence on $S^1$. This is in an analogy with
the C\'esaro summation of the Fourier series of a continuous
function where the sum of the Fourier series does not necessarily
converge at a single point of $S^1$ (see \cite{Kat}).

An interesting problem in Teichm\"uller theory is to express the
complex structure of Teichm\"uller spaces in terms of the hyperbolic
metric on the underlining surfaces. Kerckhoff proved that the
Hilbert transform $H$ of the Zygmund vector fields on $S^1$ is the
almost complex structure on the tangent space $\mathcal{Z}$ at the
basepoint of the universal Teichm\"uller space $T(\H )$ (see
Nag-Verjovsky \cite{NV}). We compute the Hilbert transform in terms
of infinitesimal shear functions (which are hyperbolic geometry
invariants) which gives a solution to computing of the almost
complex structure on the universal Teichm\"uller space $T(\H )$ in
terms of the hyperbolic geometry (see Theorem
\ref{thm:Hilbert-on-Zygmund} and Theorem
\ref{Hilbert-single-shear}).

\vskip .2 cm

\noindent {\bf Theorem 3.} {\it Let $V$ be a Zygmund bounded vector
field on $S^1$ and let $\dot{s}:\F\to\R$ be the corresponding shear
function. Then the Hilbert transform of $V$ is given by the series
$$
HV(x)=\sum_{p\in\hat{\Q}}HV_p(x)
$$
where $HV_p$ is the Hilbert transform of $V_p$ and $V_p$ is defined
as above using the shears in the fan $\F_p$. The series converges
uniformly on $S^1$.

Moreover, the Hilbert transform $H(V_p)$ of the vector field $V_p$
corresponding to the fan $\F_p=\{ e^p_n\}_{n\in\Z}$ is given by the
formula
$$
HV_p(z)=\frac{1}{2}\sum_{n=-\infty}^{\infty} \dot{s}(e_n^p)
\frac{(z-a_n^p)(z-b_n^p)}{a_n^p-b_n^p}\log
\Big{|}\frac{z-b^p_n}{z-a^p_n}\Big{|}
$$
where $a_n^p$ is the initial point and $b_n^p$ is the terminal point
of $e_n^p$.}

\vskip .2 cm

Theorem 3 gives the Hilbert transform on Zygmund vector fields on
$S^1$ in terms of infinitesimal shear functions from Theorem 2.
Penner \cite{Pe3} gave a related formula for the Hilbert transform
on $C^{5/2+\epsilon}$-smooth vector fields which are contained in
the space of Zygmund vector fields. The formula in Theorem 3 is a
double sum with the exterior sum over all all fans of $\F$ and the
interior sum over all geodesic of a single fan. We remark that a
single sum over all geodesics in $\F$ does not necessarily converge
which forces the double sum. The formula of Theorem 3 can be
explicitly given in terms of the hyperbolic geometry and we obtain
an expression for the infinitesimal shear function $H(\dot{s})$ of
the Zygmund bounded vector field $H(V)$ in terms of the hyperbolic
invariants of $\F$ and the infinitesimal shear function $\dot{s}$
(see Corollary \ref{cor:shears_of_Hilbert(V)} and Corollary
\ref{cor:shears_of_Hilbert(V)_invariant}).

We also compute the Fourier coefficients of Zygmund bounded vector
fields in terms of infinitesimal shear functions (see Theorem
\ref{thm:Fourier_vect_field}).

\vskip .2 cm

\noindent {\bf Theorem 4.} {\it Let $V:S^1\to \mathbf{C}$ be a
Zygmund vector field and let $\dot{s}:\mathcal{F}\to\R$ be the
corresponding infinitesimal shear function. Then the $n$-th Fourier
coefficient $\widehat{V}(n)$ of the vector field $V$ is given by
$$
\widehat{V}(n)=\sum_{p\in\mathcal{F}^0}\widehat{V}_p(n),
$$
where the convergence is absolute and
\begin{equation*}
\begin{split}
\widehat{V}_p(n)=\sum_{e_n=(e^{i\phi_0^n},e^{i\phi_1^n})\in\mathcal{F}_p}
\frac{\dot{s}(e_n)}{4\pi
(e^{i\phi_0}-e^{i\phi_1})}\Big{[}\frac{e^{i(2-n)\phi_1}
-e^{i(2-n)\phi_0}}{i(2-n)}- (e^{i\phi_0}+e^{i\phi_1})\times \\
\times \frac{e^{i(1-n)\phi_1}-e^{i(1-n)\phi_0}}{i(1-n)} +
e^{i(\phi_0+\phi_1)}\frac{e^{-in\phi_1}-e^{-in\phi_0}}{-in}\Big{]}
\end{split}
\end{equation*}
}

Finally, we consider the almost-complex structure on the
Teichm\"uller space $T(S)$ of a finite area hyperbolic surface $S$
with at least one puncture. Let $\tau$ be an ideal geodesic
triangulation of $S$ and let $\tilde{\tau}$ be the lift of $\tau$ to
the hyperbolic plane $\H$. Then a infinitesimal shear function
$\dot{\tilde{s}}:\tilde{\tau}\to\R$ which is invariant under the
deck transformations for the universal covering of $S$ and which
satisfies the cusp condition (see Section 8) represents a tangent
vector $V$ to $T(S)$. We obtain a formula for the infinitesimal
shear function $H(\dot{\tilde{s}})$ of the Hilbert transform of $V$
and for the Weil-Petersson metric on $T(S)$ (see Corollary
\ref{cor:shear_of_Hilbert_invariant} and Theorem
\ref{thm:WP_in_shears}).

\vskip .2 cm

\noindent {\bf Theorem 5.} {\it Let
$\dot{\tilde{s}}:\tilde{\tau}\to\R$ be the infinitesimal shear
function of a Zygmund bounded vector field $V$ invariant under a
co-finite group $G$, where $\tilde{\tau}$ is an ideal geodesic
triangulation of $\H$ invariant under $G$. Then the infinitesimal
shear function
$$
H(\dot{\tilde{s}}):\tilde{\tau}\to\R
$$
of the vector field $H(V)$ obtained by taking the Hilbert transform
of $V$ is given by
\begin{equation*}
H(\dot{\tilde{s}})((b,d))=\sum_{n\in\mathbf{N}}\dot{\tilde{s}}(e_n)\Delta_{b,d}(e_n)
\end{equation*}
where $(b,d)\in\tilde{\tau}$ is the common boundary side of the two
complementary triangles of $\tilde{\tau}$ with vertices $(a,b,d)$
and $(b,c,d)$, $\tilde{\tau}=\{e_n\}_{n\in\mathbf{N}}$,
\begin{equation*}
\begin{split}
\Delta_{b,d}(e_n)=\sinh^2\Big{(}\frac{\delta_{b,c}}{2}\Big{)}\log\coth^2
\Big{(}\frac{\delta_{b,c}}{2}\Big{)}+
\cosh^2\Big{(}\frac{\delta_{a,d}}{2})
\log\coth^2 \Big{(}\frac{\delta_{a,d}}{2}\Big{)}-\\
\cosh^2\Big{(}\frac{\delta_{a,b}}{2}\Big{)}\log\coth^2
\Big{(}\frac{\delta_{a,b}}{2}\Big{)}-
\cosh^2\Big{(}\frac{\delta_{c,d}}{2}\Big{)}\log\coth^2
\Big{(}\frac{\delta_{c,d}}{2}\Big{)}
\end{split}
\end{equation*}
and $\delta_{b,c}$ is the distance between between the geodesic
$e_n$ and the geodesic $(b,c)$ with analogous definition in other
cases $\delta_{a,d},\delta_{a,b},\delta_{c,d}$. If $e_n$ has at
least one endpoint in $(a,b,c,d)$ then $\Delta_{b,d}(e_n)$ is given
appropriate interpretation as in Section 6.

Moreover, the Weil-Petersson metric pairing of the two tangent
vectors $v_1,v_2$ to $T(S)$ represented by two infinitesimal shear
functions $\dot{s}_1,\dot{s}_2:\tau\to\R$ is given by
$$
g_{WP}(v_1,v_2)=2i(\dot{s}_1,H(\dot{s}_2))
$$
where $H(\dot{s}_2)$ is the projection of $H(\dot{\tilde{s}}_2)$
onto $\tau$ and $i(\cdot ,\cdot )$ is Thurston's algebraic
intersection number.}

\vskip .2 cm

We point out that the expression in the above theorem which
represents the almost complex structure on the tangent space in
terms of infinitesimal shear functions is given by a single sum
unlike for arbitrary Zygmund bounded vector fields where the sum is
double. This feature appears because the infinitesimal shear
functions are invariant under the action of a cofinite Fuchsian
group $G$. Finally, we use the above expression of the almost
complex structure in terms of shear coordinates to conclude that the
almost complex structure on $T(S)$ extend continuously to the
boundary Teichm\"uller spaces of $T(S)$ obtained by pinching some
closed curves on $S$ (see Section 10). Together with the result of
J. Roger \cite{Ro} which establishes that the Thurston's algebraic
intersection number extend by continuity to the boundary of $T(S)$,
we obtain the well-known result of Masur \cite{mas} that the
Weil-Petersson metric extends by continuity to the boundary of
$T(S)$. We also point out that G. Riera \cite{Rier} obtained a
formula for the Weil-Petersson metric on Teichm\"uller spaces of
closed surfaces in terms of hyperbolic invariants.

\section{Zygmund vector fields and the tangent space to $T(\H )$}

Let $\H =\{Im(z)>0\}$ be the upper half-plane model of the
hyperbolic plane with the hyperbolic metric given by $\rho
(z)=\frac{|dz|}{Im(z)}$. Then $\hat{\R}=\R\cup\{ \infty\}$ is the
ideal boundary of $\H$ and $\hat{\R}$ is homeomorphic to the unit
circle $S^1$. A continuous map
$$V:\R\to \R$$ such that
\begin{equation}
\label{eq:V(x)_at_infinity}
\lim_{x\to\infty}\frac{V(x)}{x^2}=\lim_{x\to
-\infty}\frac{V(x)}{x^2}<\infty
\end{equation}
represents a {\it (continuous) vector field on} $\hat{\R}$. $\R$ is
a chart of $\hat{\R}$ and the condition (\ref{eq:V(x)_at_infinity})
guarantees that $V$ extends by continuity to a well-defined vector
field at $\infty\in\hat{\R}$.

A vector field $V:\R\to\R$ on $\hat{\R}$ is said to be {\it Zygmund}
if there exists $C>0$ such that both $V(x)$ and $x^2V(1/x)$ satisfy
\begin{equation}
\label{eq:Zygmund_vector_field} |V(x+t)+V(x-t)-2V(x)|\leq Ct
\end{equation}
for $x\in [-2,2]$ and $t\in [0,1]$. If $V$ is Zygmund, then $V$ can
be normalized by adding a quadratic polynomial in $x$ such that
$V(0)=V(1)=0$ and $V(x)=O(x\log |x|)$ as $|x|\to\infty$ (see
\cite{GL}). We are mainly interested in Zygmund vector fields on
$\hat{\R}$. A quadratic polynomial in $x$ describes an infinitesimal
change in $PSL_2(\R )$ which is a trivial change in our context of
the universal Teichm\"uller space $T(\H )$. Alternatively, a vector
field $V:\R\to\R$ on $\hat{\R}$ is Zygmund if it satisfies
(\ref{eq:Zygmund_vector_field}) for all $x\in\R$ and $t>0$. The {\it
Zygmund norm} of a Zygmund vector field $V:\R\to\R$ is the smallest
constant $C$ for which (\ref{eq:Zygmund_vector_field}) remains in
force.

Let $a,b,c,d\in\hat{\R}$ be four distinct points. The {\it
cross-ratio} $cr(a,b,c,d)$ is defined by
$$
cr(a,b,c,d)=\frac{(c-b)(d-a)}{(b-a)(d-c)}.
$$
The {\it cross-ratio norm} $\| V\|_{cr}$ of a vector field
$V:\R\to\R$ on $\hat{\R}$ is defined by
$$
\|
V\|_{cr}=\sup_{cr(a,b,c,d)=1}\Big{|}\frac{V(c)-V(b)}{c-b}+\frac{V(d)-V(a)}{d-a}
-\frac{V(b)-V(a)}{b-a}-\frac{V(d)-V(c)}{d-c}\Big{|}.
$$

A sequence $\{ V_n\}_{n\in\mathbf{N}}$ of Zygmund vector fields
converges to a Zygmund vector field $V$ in the topology of the
Zygmund norm (the cross-ratio norm) if and only if the Zygmund norm
(the cross-ratio norm) of $V-V_n$ goes to zero as $n\to\infty$. The
convergence in the Zygmund norm is equivalent to the convergence in
the cross-ratio norm \cite{Hu}, \cite{GL}.

Gardiner and Sullivan \cite{GS}, and Riemann \cite{Riem} showed that
a normalized vector field $V:\R\to\R$ is Zygmund if and only if
there exists a differentiable path of quasisymmetric maps
$$h_t:\hat{\R}\to \hat{\R}$$ for $|t|<\epsilon$ and $\epsilon >0$
such that $h_t:0,1,\infty \mapsto 0,1,\infty$ and
$$\frac{d}{dt}h_t|_{t=0}=V$$ on $\R$.

We recall that the universal Teichm\"uller space $T(\H )$ consists
of all quasisymmetric maps $h:\hat{\R}\to\hat{\R}$ such that
$h:0,1,\infty\mapsto 0,1,\infty$. Thus the space $\mathcal{Z}$ of
normalized Zygmund maps is identified with the tangent space
$T_{id}T(\H )$ of the universal Teichm\"uller space $T(\H )$ at the
identity ({\it basepoint}) $id\in T(\H )$.

\section{Complex structure on $T(\H )$ and Hilbert transform on Zygmund vector fields}

Let $V:\R\to\R$ be a normalized Zygmund vector field $V$ on
$\hat{\R}$, namely $V$ is a tangent vector to $T(\H )$ at the
basepoint $id\in T(\H )$. Let $\mu\in L^{\infty}(\H )$ be a Beltrami
differential representing the tangent vector $V\in T_{id}T(\H )$.
Denote by $[\mu ]$ the equivalence class of Beltrami differentials
representing the same tangent vector as $\mu$. We have \cite{A},
\cite{GL} a standard formula
$$
V(x)=-\frac{2}{\pi}Re\int_{\H}R(x,\zeta )\mu (\zeta )d\xi d\eta ,
$$
where $R(x,\zeta )=\frac{x(x-1)}{\zeta (\zeta -1)(\zeta -x)}$,
$\zeta =\xi +i\eta\in\H$ and $x\in\R$.

The complex structure operator $J:T_{id}T(\H )\to T_{id}T(\H )$ is
given by $J([\mu ])=[i\mu]$, where $[\mu ]$ is the tangent vector
represented by $\mu\in L^{\infty}(\H )$ and $i\in\mathbf{C}$ is the
imaginary unit. The above formula gives that the complex structure
on $V$ is
$$
J(V)(x)=-\frac{2}{\pi}Re\int_{\H}R(x,\zeta )i\mu (\zeta )d\xi d\eta
$$
or, equivalently
$$
J(V)(x)=\frac{2}{\pi}Im\int_{\H}R(x,\zeta )\mu (\zeta )d\xi d\eta .
$$

Let $\tilde{V}(z)$ for $z\in \H$ be the Beurling-Ahlfors extension
of $V$ into $\H$ (see \cite{GS}, \cite[pages 316,317]{GL}). Then we
have that $\bar{\partial}\tilde{V}(z)=\mu (z)$ a.e. in $\H$ (see
\cite{GL}). By Stoke's theorem, we get
$$
\frac{2i}{\pi}\int_{\H}R(x,\zeta )\mu(\zeta )d\xi d\eta=p.v.
\frac{1}{\pi}\int_{\R} R(x,\xi )V(\xi )d\xi -iV(x).
$$

By taking the real parts in the above equality, we obtain that the
complex structure $J$ on $V$ is
$$
J(V)(x)=-\frac{1}{\pi} p.v. \int_{\R} R(x,\xi )V(\xi )d\xi
$$
for $x\in\R$ which equals the {\it Hilbert transform} of $V$
\begin{equation}
\label{eq:Hilbert_transform} H(V)(x)=-\frac{1}{\pi} p.v. \int_{\R}
R(x,\xi )V(\xi )d\xi
\end{equation}

Kerckhoff was first to observe that the Hilbert transform gives the
almost-complex structure operator on $T_{id}T(\H )$ and the proof of
this statement first appears in \cite{NV}. We used the discussion in
\cite{Ga}. The Hilbert transform preserves the space of normalized
Zygmund vector fields on $\hat{\R}$ (see \cite{Zyg}, \cite{Ga}).

Let $V_{(0,\infty )}:\R\to\R$ be the {\it elementary shear vector
field} for the geodesic $(0,\infty )$ defined by
\begin{equation}\label{eq:simple_shear_vector_field}
V_{(0,\infty)}(x)=\left\{
\begin{array}l
x,\ \ \ \mbox{for}\ x>0\\
0,\ \ \ \mbox{for}\ x\leq 0
\end{array}
\right.
\end{equation}
Then the Hilbert transform of $V_{(0,\infty )}$ is given by
\begin{equation}
\label{eq:simple_shear_Hilbert}
HV_{(0,\infty
)}(x)=\frac{1}{\pi}x\log |x|
\end{equation}
for all $x\in\R$.

Some more elementary integration gives
\begin{equation}
\label{eq:simple_shear_Hilbert1} \begin{split} H(V_{(a,\infty
)})(x)=\frac{1}{\pi}(x-a)\log |x-a|+ \frac{(a-1)\log |a-1|}{\pi}x-\\
-\frac{a\log a}{\pi}(x-1)\ \ \ \ \ \ \ \ \ \ \ \ \ \ \ \ \ \ \ \ \ \
\ \ \ \ \ \ \ \ \ \ \ \ \ \
\end{split}
\end{equation}
for all $x\in\R$, where $a\in\R$, $a>0$, and the {\it elementary
shear vector field} $V_{(a,\infty )}$ for the geodesic  $(a,\infty
)$ is defined by
\begin{equation*}
V_{(a,\infty)}(x)=\left\{
\begin{array}l
x-a,\ \ \mbox{for}\ x>a\\
0,\ \ \ \ \ \ \ \ \mbox{for}\ x\leq a
\end{array}
\right.
\end{equation*}

For $a\in\R$, $a<0$, the {\it elementary shear vector field} for the
geodesic $(-\infty ,a)$ is defined by
\begin{equation*}
V_{(-\infty ,a)}(x)=\left\{
\begin{array}l
-(x-a),\ \mbox{for}\ x<a\\
0,\ \ \ \ \ \ \ \ \ \ \ \mbox{for}\ x\geq a
\end{array}
\right.
\end{equation*}

Then \begin{equation} \label{eq:simple_shear_Hilbert2} \begin{split}
H(V_{(-\infty ,a
)})(x)=-\frac{1}{\pi}(x-a)\log |x-a|- \frac{(a-1)\log |a-1|}{\pi}x+\\
+\frac{a\log |a|}{\pi}(x-1)\ \ \ \ \ \ \ \ \ \ \ \ \ \ \ \ \ \ \ \ \
\ \ \ \ \ \ \ \ \ \ \ \ \ \ \
\end{split}
\end{equation}

Let $a,b\in\R$ and $a<b$. For the {\it elementary shear vector
field}
\begin{equation*}
V_{(a,b)}(x)=\left\{
\begin{array}l
\frac{(x-a)(x-b)}{a-b},\ \mbox{for}\ a<x<b\\
0,\ \ \ \ \ \ \ \ \ \mbox{otherwise}
\end{array}
\right.
\end{equation*}
we obtain the Hilbert transform
\begin{equation}
\label{eq:simple_shear_Hilbert3}
\begin{split}
H(V_{(a,b)})(x)=\frac{1}{\pi}\frac{(x-a)(x-b)}{a-b}\log
\Big{|}\frac{x-b}{x-a}\Big{|}-\ \ \ \ \ \ \ \ \ \ \ \ \ \ \ \\
-x\frac{(1-a)(1-b)}{\pi (a-b)}\log \Big{|}
\frac{b-1}{a-1}\Big{|}+(x-1)\frac{ab}{\pi
(a-b)}\log\Big{|}\frac{b}{a}\Big{|}
\end{split}
\end{equation}
for all $x\in\R$.

\section{Shear parametrization of Zygmund maps}

Let $V:\R\to\R$ be a normalized Zygmund vector field. Namely $V(x)$
satisfies (\ref{eq:Zygmund_vector_field}) and it vanishes at $0$,
$1$ and $\infty$, where vanishing at $\infty$ means that the rate of
growth of $|V(x)|$ is $|x|\log |x|$ as $x\to\pm\infty$ (see
\cite{GL}). Let $h_t:\hat{\R}\to\hat{\R}$, $|t|<\epsilon$, be a path
of quasisymmetric maps such that $h_0=id$, $h_t$ fixes $0$, $1$ and
$\infty$ for each $t$, $|t|<\epsilon$, and
$$
\frac{d}{dt}h_t(x)|_{t=0}=V(x)
$$
for all $x\in\R$. In addition, we assume that the quasisymmetric
constants of $h_t$ are bounded by $1+Ct$ for some constant $C$ and
for each $t$, $|t|<\epsilon$ (see \cite{GS}, \cite{GL}).

Given two ideal hyperbolic triangles $\Delta_1$ and $\Delta_2$ with
common boundary side $e$ and disjoint interiors, the shear of
$\Delta_1$ and $\Delta_2$ is the signed hyperbolic distance between
the projections of third vertices of $\Delta_2$ and $\Delta_1$ onto
$e$ with the orientation of $e$ as a part of the boundary of
$\Delta_1$. A homeomorphism $h:\hat{\R}\to\hat{\R}$ maps the Farey
tesselation $\F$ onto a tesselation $h(\F )$. Given $e\in\F$, let
$\Delta_1^e$ and $\Delta_2^e$ be the two complementary triangles of
$h(\F )$ with common boundary side $h(e)$. We define a shear
function $s:\F\to\R$ of $h$ by setting $s(e)$ to be equal to the
shear of the triangles $\Delta_1^e$ and $\Delta_2^e$.

Let $s_t:\F\to\R$ be the path of shear functions corresponding to
the path $h_t$ of quasisymmetric maps. It follows that
$$
\dot{s}(e):=\frac{d}{dt}s_t(e)|_{t=0}
$$
is well-defined for each $e\in\F$ by the differentiability of $h_t$.
The function $$\dot{s}:\F\to\R$$ associated to $V$ is called the
{\it infinitesimal shear function} of $V$ and it uniquely determines
$V$ when $V$ is normalized to vanish at $0$, $1$ and $\infty$.

A {\it fan of geodesics} in $\F$ with {\it tip}
$p\in\hat{\mathbf{Q}}$ consists of all geodesics of $\F$ with a
common endpoint $p$. We orient each geodesic of the fan with tip $p$
such that $p$ is its initial point. Then for two geodesics $e$ and
$e'$ of the fan with tip $p$ we define $e<e'$ if $e'$ is to the left
of $e$ and we define $e'<e$, otherwise. This induces a one-to-one
correspondence of the fan $\{ e_n\}_{n\in\Z}$ with integers $\Z$
such that $e_n$ is adjacent to $e_{n+1}$ and $e_n<e_{n+1}$ for all
$n\in\Z$. Under the above assumptions on $h_t$, there exists $C>0$
such that
\begin{equation}
\label{eq:quasisymmetry-shear} \frac{1}{1+Ct}\leq e^{s_t(m)}
\frac{1+e^{s_t(m+1)}+\cdots +e^{s_t(m+1)+\cdots
+s_t(m+k)}}{1+e^{-s_t(m-1)}+\cdots +e^{-s_t(m-1)-\cdots
-s_t(m-k)}}\leq 1+Ct
\end{equation}
for all $m,k\in\Z$, where $s_t(m):=s_t(e_m)$ (see \cite{Sa2}). We
have the expansion
$$
s_t(m)=\dot{s}(m)t+o_m(t)
$$
where $o_m(t)/t\to 0$ as $t\to 0$ and constant in $o_m(t)$ depends
on $m$. From (\ref{eq:quasisymmetry-shear}) and the above expansion,
we obtain
\begin{equation}
\begin{split}
\label{eq:zygmund-condition-shears}
\Big{|}\dot{s}(m)+\frac{k-1}{k}[\dot{s}(m+1)+\dot{s}(m-1)]+
\frac{k-2}{k}[\dot{s}(m+2) +\\ +\dot{s}(m-2)]+  \cdots
+\frac{1}{k}[\dot{s}(m+k)+\dot{s}(m-k)]\Big{|}\leq C
\end{split}
\end{equation}
for all $m,k\in\Z$, for all fans $\{ e_n\}_{n\in\Z}$ of $\F$ and for
the fixed constant $C>0$ from the above. Thus
$(\ref{eq:zygmund-condition-shears})$ is a necessary condition for
$\dot{s}:\F\to\R$ to be the infinitesimal shear function of a
Zygmund map $V$. We show that (\ref{eq:zygmund-condition-shears}) is
also a sufficient condition. Thus the space $\mathcal{Z}$ of all
Zygmund vector fields normalized to vanish at $0$, $1$ and $\infty$
is parameterized by the set of functions $\dot{s}:\F\to\R$ which
satisfy (\ref{eq:zygmund-condition-shears}) in all fans for some
constant $C>0$ (which is the same for all fans).

\begin{theorem}
\label{thm:zygmund-parametrization-shears} Let $\dot{s}:\F\to\R$ be
an arbitrary function. Then $\dot{s}$ is the infinitesimal shear
function of a normalized Zygmund vector field $V:\R\to\R$ if and
only if $\dot{s}$ satisfies (\ref{eq:zygmund-condition-shears}) for
all fans in $\F$ with a fixed constant $C$.
\end{theorem}

\begin{proof}
We proved above that (\ref{eq:zygmund-condition-shears}) is a
necessary condition for $V$ to be a Zygmund vector field. It remains
to prove that (\ref{eq:zygmund-condition-shears}) is also a
sufficient condition for $\dot{s}$ to be induced by a Zygmund vector
field $V:\R\to\R$. We first show that $\dot{s}$ is induced by a
continuous vector field $V:\R\to\R$.

\vskip .2 cm

Let $\Delta_0$ be a hyperbolic triangle in the upper half-plane $\H$
with vertices $0$, $1$ and $\infty$. We orient each edge in $\F$ to
the left as seen from $\Delta_0$. Consider a fan of geodesics
$\F_p=\{ e_n^p\}_{n\in\Z}$ with tip $p\in\hat{\Q}$. The geodesics of
the fan $\F_p$ are divided into two groups with respect to the
orientation given by $\Delta_0$: the geodesics whose initial point
is $p$ and the geodesics whose terminal point is $p$. We enumerate
the fan with tip $p$ by $\{ e_n^p\}_{n\in\Z}$ such that $e_n^p$ is
adjacent to $e_{n+1}^p$, and that $e_0^p$ has initial point $p$ and
$e_1^p$ has terminal point $p$. This enumeration is unique. In the
case $p=\infty$, we have that $e_n^{\infty}$ is a geodesic with
endpoints $n$ and $\infty$.

Assume that function
$$
\dot{s}:\Z\to\R
$$
satisfies (\ref{eq:zygmund-condition-shears}). We define
\begin{equation}
\label{eq:fan-zygmund} V_{\dot{s}}^{\infty}(x)=\left\{
\begin{array}l
0\ \ \ \ \ \ \ \ \ \ \ \ \ \ \ \ \ \ \ \ \ \ \ \ \ \ \ \ \ \ \ \ \ \ \ \ \ \ \ \ ; \ x\in [0,1]\\
{\dot{s}}(1)(x-1)+\cdots +{\dot{s}}(n)(x-n)\ ;\ x\in
(n,n+1] \\
-{\dot{s}}(0)x-\cdots -{\dot{s}}(-n)(x+n) \ \ \ \ ;\ x\in [-n-1,-n)
\end{array}\right .
\end{equation}
for $n\in\N$.

\begin{proposition} \label{prop:claim1} Under the above notation,
the vector field $V_{\dot{s}}^{\infty}:\R\to\R$ is Zygmund with the
Zygmund norm $2C+18\|\dot{s}\|_{\infty}\leq 20C$, where $C$ is given
in (\ref{eq:zygmund-condition-shears}).
\end{proposition}

\begin{proof} It is clear that $V_{\dot{s}}^{\infty}$ is continuous. We
consider $x\in\R$ and $t\geq 2$. Let $m\in\Z$ and $k\in\N$ be such
that $m\leq x\leq m+1$ and $m+k\leq x+t\leq m+k+1$. Then we have
$k-1\leq t\leq k+1$ and $m-k-1\leq x-t\leq m-k+2$. A simple
estimation shows that
\begin{equation*}
\begin{split}
|V_{\dot{s}}^{\infty}(x+t)+V_{\dot{s}}^{\infty}(x-t)-2V_{\dot{s}}^{\infty}(x)|\leq
|V_{\dot{s}}^{\infty}(m+k) + V_{\dot{s}}^{\infty}(m-k)-\\
-2V_{\dot{s}}^{\infty}(m)|
+6\|{\dot{s}}\|_{\infty}k+3\|\dot{s}\|\leq C_1t,
\end{split}
\end{equation*}
where $C_1=2C+18\|{\dot{s}}\|_{\infty}$ because $k\geq 1$ implies
$2t\geq k$ and condition (\ref{eq:zygmund-condition-shears}) is
equivalent to $|V_{\dot{s}}^{\infty}(m+k) +
V_{\dot{s}}^{\infty}(m-k) -2V_{\dot{s}}^{\infty}(m)|\leq Ck$.

For $0<t<2$, we have that $V_{\dot{s}}^{\infty}$ is piecewise linear
on each interval $[x-t,x+t]$ with at most three points where it is
not differentiable and with the bound $\|{\dot{s}}\|_{\infty}$ on
the difference between the two adjacent slopes. Thus
$|V_{\dot{s}}^{\infty}(x+t)+V_{\dot{s}}^{\infty}(x-t)-2V_{\dot{s}}^{\infty}(x)|\leq
C_1t$ in this case as well. Thus $V_{\dot{s}}^{\infty}$ is Zygmund
bounded with the Zygmund norm $C_1$.
\end{proof}

\vskip .2 cm

\noindent {\it Continuation of the proof of Theorem
\ref{thm:zygmund-parametrization-shears}.} We introduce an order on
$\hat{\Q}$ using the Farey tesselation. Namely, $0$ and $\infty$ in
$\hat{\Q}$ are said to be of order $1$. The numbers $1$ and $-1$ are
said to be of order $2$. If the maximum of the orders of $a/b$ and
$c/d$ is $n$ and they are adjacent on $\hat{\R}$ (when we consider
only numbers in $\hat{\Q}$ of order at most $n$), then the mediant
$(a+c)/(b+d)$ is of the order $n+1$. In this fashion, we obtain the
whole $\hat{Q}$. There are two order $1$ and two order $2$ numbers
in $\hat{Q}$. For $n\geq 2$, there are $2^{n-1}$ numbers of order
$n$.

The reference triangle $\Delta_0$ has vertices $0$, $1$ and
$\infty$. Fix $p\in\hat{\Q}$ and let $\F_p=\{ e_n^p\}_{n\in\Z}$ be
the fan with tip $p$. Recall that $e_n^p$'s are oriented to the left
as seen from $\Delta_0$, and that $e_0^p$ has the initial point $p$
and $e_1^p$ has the terminal point $p$. Let $a$ be the terminal
endpoint of $e_0^p$ and let $B\in PSL_2(\Z )$ be the unique element
such that $B(0)=a$ and $B(\infty )=p$. Let $V_{\dot{s}_p}^{\infty}$
be given by formula (\ref{eq:fan-zygmund}) where
$$\dot{s}_p:\Z\to\R$$ is defined by
$$\dot{s}_p(n)=\frac{1}{2}\dot{s}(e^p_n)$$
for $n\in\Z$. (Note that we use $\frac{1}{2}\dot{s}$ because each
geodesic of $\F$ appears in two different fans.) Then
$V_{\dot{s}_p}^{\infty}$ is a Zygmund vector field with the Zygmund
constant $10C$  by Proposition \ref{prop:claim1}.

Define
$$
V_p:=B^{*}V_{\dot{s}_p}^{\infty}
$$
Then $V_p$ is a Zygmund vector field corresponding to the shear
function which agrees with $\frac{1}{2}\dot{s}$ at $\F_p=\{
e_n^p\}_{n\in\Z}$ and which is zero on $\F\setminus\F_p$. If $p\neq
0,1$ then $V_p$ is normalized to be zero on $(-\infty ,B(0))\cup
(B(1),+\infty )$ if $B(0)<B(1)$, otherwise $V_p$ is zero on
$(-\infty ,B(1))\cup (B(0),+\infty )$ if $B(1)<B(0)$. In both cases
$V_p$ vanishes at $\infty$. It is also true that $V_p$ vanishes at
$0$ and $1$ because geodesics in $\F_p$ do not separate $0$ and $1$
from $\Delta_0$. Thus $V_p$ is a normalized Zygmund vector field.

\vskip .2 cm

We define a formal series
$$
V(x):=\sum_{p\in\hat{\Q}}V_p(x)
$$
for $x\in\R$. It is clear that $V$ is well-defined on $\hat{\Q}$
because only finitely many terms $V_p(x)$ in the above series are
non-zero. (If we knew that $V$ is a tangent vector to some path of
homeomorphisms, then the corresponding infinitesimal shear function
would be equal to $\dot{s}:\F\to\R$. However, we still need to show
that $V$ is defined on all of $\R$ and that it is Zygmund.)

\begin{proposition}
\label{prop:claim2} The series
\begin{equation}
\label{eq:series_for_V(x)} V(x):=\sum_{p\in\hat{\Q}}V_p(x)
\end{equation}
converges absolutely and uniformly on compact subsets of $\R$.
\end{proposition}

\begin{proof} Let $x\in\R$. By Proposition \ref{prop:claim1} and the above,
the Zygmund norm of each $V_{\dot{s}_p}^{\infty}$ is bounded by
$C_1=10C$. This implies that the cross-ratio norm of each
$V_{\dot{s}_p}^{\infty}$ is bounded by $C_2=C_2(C_1)$ (see
\cite{GS}, \cite{GL}). Thus the Zygmund norm of $V_b$ is bounded by
$C_2$ because $V_p$ is normalized to vanish at $0$, $1$ and
$\infty$.

Let $r_x$ be the geodesic ray from the center of $\Delta_0$ to
$x\in\R$. Consider the set of fans whose geodesics intersect $r_x$.
If $r_x$ intersects $\F_p$ and the Farey order of $p$ is at least
$2$, then $r_x$ does not intersect any other fan with the same
order. On the other hand, $r_x$ intersects both $\F_{0}$ and
$\F_{\infty}$ if and only if $r_x$ intersects the geodesic with
endpoints $0$ and $\infty$. We order the set of all fans
intersecting $r_x$ by the Farey order of their tips. The only
ambiguity is if both $\F_{0}$ and $\F_{\infty}$ intersect $r_x$ in
which case we set $\F_{0}<\F_{\infty}$. This gives us a linear order
on the set of fans intersecting $r_x$ and we can put them in a
sequence $\{\F_{p_n}\}_{n\in\N}$ such that $\F_{p_n}<\F_{p_{n+1}}$.
Consider the ordered subset $\{ e_{i_1}^{p_n},\ldots
,e_{i_{k(n)}}^{p_n}\}$ of geodesics of $\F_{p_n}$ that intersect
$r_x$. Then either $i_1=0$ and $i_1>i_2>\cdots >i_k$, or $i_1=1$ and
$i_1<i_2<\cdots <i_k$.

Consider the fans $\F_{p_n}$ and $\F_{p_{n+2}}$. The last geodesic
$e^{p_n}_{i_{k(n)}}$ of the fan $\F_{p_n}$ which intersects $r_x$
does not share an endpoint with the first geodesic
$e^{p_{n+2}}_{i_{1}}$ of the fan $\F_{p_{n+2}}$ which intersects
$r_x$. Thus the distance between $e^{p_n}_{i_{k(n)}}$ and
$e^{p_{n+2}}_{i_{1}}$ is $\delta =\log (1+\sqrt{2})^2$.

Assume that $n\geq 3$. Then no geodesic of the fan $\F_{p_n}$ has
$\infty$ as an endpoint. Moreover, the Farey order of $p_n$ is at
least $3$. Let $e_{i_1}^{p_n}=(a_n,b_n)$ be the first geodesic of
the fan $\F_{p_n}$ that intersects $r_x$.   Note that
$V_{p_n}(a_n)=V_{p_n}(b_n)=0$ by the normalization of $V_{p_n}$. The
distance of the center of the triangle $\Delta_0$ to $e_{i_1}^{p_n}$
is at least $\frac{n-2}{2}\delta$ by the above considerations. An
elementary hyperbolic geometry implies that $|a_n-b_n|$ is of the
order $e^{-\frac{(n-2)\delta}{2}}$ for $n\in\N$, where the constant
depends on a compact subset of $\R$ in which $x$ lies. Since
$V_{p_n}$ has $\epsilon\log\frac{1}{\epsilon}$ modulus of continuity
with the constant depending on its Zygmund norm and since the
support of $V_{p_n}$ is the interval in $\R$ with endpoints $a_n$
and $b_n$, for $x$ in a compact subset of $\R$, we obtain
$$
\sum_{p\in\hat{\Q}}|V_p(x)|=\sum_{n\in\N}|V_{p_n}(x)|\leq
C_0+C_3\sum_{n\in\N} ne^{-\frac{n\delta}{2}}<\infty
$$
where $C_0=|V_{p_1}(x)|+|V_{p_2}(x)|+|V_{p_3}(x)|$, and
$C_3=C_3(C_2)$ is a constant which depends on the bound $C_2$ on the
Zygmund norm of $V_{p_n}$'s.
\end{proof}

\vskip .2 cm

\begin{remark}
\label{rmk:extension_to_infinity} We showed that the formal series
(\ref{eq:series_for_V(x)}) converges to a continuous vector field on
$\R$ which is normalized to vanish at $0$ and $1$ because each $V_p$
vanishes at $0$ and $1$. To see that $V(x)$ extends to a
(continuous) normalized vector field on $\hat{\R}$ we need to show
that $V(x)=o(x^2)$ as $x\to\pm\infty$. In the proof of Proposition
\ref{prop:claim2} we showed that $V(x)$ (for $x$ in a compact subset
of $\R$) is bounded in terms of the Zygmund bound on $V_p$'s. The
push-forward of the vector fields $V_p$ under the translation
$T:x\mapsto x+a$, $a\in\R$, satisfies $\| V_p\|_{\infty}=\|
T^{*}(V_p)\|_{\infty}$ because $T^{*}(V_p)(x)=V_p(x-a)$. Let
$\sum'V_{p_n}(x)$ be the sum (\ref{eq:series_for_V(x)}) without
terms $V_{p_1}$, $V_{p_2}$ and $V_{p_3}$. Note that each $V_{p_n}$
in the sum $\sum'V_{p_n}(x)$ does not have a positive shear on a
geodesic of $\F$ with one endpoint $\infty$. If $x\in [k,k+1]$ for
$k\in\Z$, define $T_k(x)=x-k$. Then $(T_k)^{*}(\sum'V_{p_n}(x))$ is
equal to $V_k'(x-k)$, where the infinitesimal shear function of the
vector field $V'_k$ is the push-forward by $T_k$ of infinitesimal
shear function $\dot{s}$. Since $V'_k$ is bounded on compact subsets
of $\R$ by Proposition \ref{prop:claim2}, since $x-k\in [0,1]$ and
since $\sum'V_{p_n}(x)=V_k'(x-k)$, it follows that $\sum'V_{p_n}(x)$
is bounded on $\R$. Since $V_{p_1}$, $V_{p_2}$ and $V_{p_3}$ are
normalized to vanish at $\infty$ and $\sum'V_{p_n}(x)$ is bounded on
$\R$, it follows that $V(x)=o(x^2)$ as $x\to\pm\infty$. Thus
(\ref{eq:series_for_V(x)}) defines a normalized vector field on
$\hat{\R}$.
\end{remark}

\begin{remark}
The sum (\ref{eq:series_for_V(x)}) defines a continuous vector field
$V:\R\to\R$ on $\hat{\R}$ which vanishes at $0$, $1$ and $\infty$.
This is achieved by normalizing each vector field $V_b$ in the same
fashion. To normalize $V(x)$ to vanish at the endpoints
$a_1,a_2,a_3$ of another complementary triangle $\Delta$ of $\F$, we
start the construction of each $V_p$ with respect to the triangle
$\Delta$. Equivalently, we need to add a quadratic polynomial to the
formula (\ref{eq:series_for_V(x)}) in order to obtain a vector field
with the same infinitesimal shear function $\dot{s}$ which vanishes
at $a_1,a_2,a_3$. The new vector field is of the order $O(x^2)$ as
$x\to\pm\infty$.
\end{remark}

\begin{remark}
\label{rmk:vect_field_nonFarey} Let $B\in PSL_2(\R )$ and let
$\F^{*}=B(\F )$. Then $\dot{s}\circ B^{-1}:\F^{*}\to\R$ is a shear
function on $\F^{*}$. Fix a complementary triangle $\Delta^{*}$ in
$\F^{*}$. Then there exists $B_1\in PSL_2(\R )$ such that
$B_1(\Delta_0)=\Delta^{*}$ and $B_1(\F )=\F^{*}$. We define
$$
W(x):=\sum_{p\in B_1(\hat{\Q})}W_p(x)
$$
where $W_p=(B_1)^{*}(V_{B_1^{-1}(p)})$. Then the series for $W(x)$
converges absolutely and uniformly on compact subsets of
$\R\setminus \{B_1(\infty )\}$ because each term is the push-forward
of the series (\ref{eq:series_for_V(x)}) by the same M\"obius map
$B_1\in PSL_2(\R )$. Vector field $W:\R\to\R$ is normalized to
vanish at the vertices of $\Delta^{*}$ and its infinitesimal shear
function is $\dot{s}\circ B_1^{-1}:\F^{*}\to\R$.
\end{remark}

\vskip .2 cm

\noindent {\it We continue the proof of Theorem
\ref{thm:zygmund-parametrization-shears}.} Proposition
\ref{prop:claim2} implies that $V(x)$ is a continuous vector field
on $\R$ because each $V_p(x)$ is a continuous vector field and the
convergence is uniform. We showed in Remark
\ref{rmk:extension_to_infinity} that $V(x)$ extends to a normalized
vector field on $\hat{\R}$. We need to show that $V(x)$ is a Zygmund
vector field. We show a slightly stronger statement which will be
used later on.

\begin{proposition}
\label{prop:Zyg_bdd_uniform_in_shears} Let $C>0$ be fixed. If
$$\dot{s}:\mathcal{F}\to\mathbf{R}$$ satisfies
(\ref{eq:zygmund-condition-shears}) with constant $C$, then there
exists $M=M(C)>0$ such that the vector field
$$V:\mathbf{R}\to\mathbf{R}$$ induced by $\dot{s}$ is $M(C)$-Zygmund
bounded.
\end{proposition}

\begin{proof}
Assume on the contrary that there exists a sequence
$\dot{t}_n:\mathcal{F}\to\mathbf{R}$ of infinitesimal shear
functions satisfying (\ref{eq:zygmund-condition-shears}) with the
fixed constant $C>0$ such that the cross-ratio norms of the
corresponding vector fields $W_n:\R\to\R$ converge to infinity as
$n\to\infty$. Then there exists a sequence $\{
Q_n=(a_n,b_n,c_n,d_n)\}$ of quadruples on $\hat{\R}$ with
$cr(Q_n)=\frac{(c_n-a_n)(d_n-b_n)}{(d_n-a_n)(c_n-b_n)}=2$ such that
$|W_n[Q_n]|\to\infty$ as $n\to\infty$. (If for a single shear
function $\dot{t}:\F\to\R$ the induced continuous vector field $W$
is not Zygmund, then there exists a sequence $Q_n$ as above such
that $|W[Q_n]|\to \infty$ as $n\to\infty$. We replace $W_n$ with $W$
and the argument below is the same.) We seek a contradiction.

Let $Q^{*}:=(-e^{\delta},-1,1,e^{\delta})$ with $\delta =\log
(1+\sqrt{2})^2$ and note that $cr(Q^{*})=2$. Then there exists a
unique M\"obius map $B_n$ such that $B_n(Q_n)=Q^{*}$. Define
$$
V_n:=(B_n)^{*}W_n .
$$
Then
$$
|V_n[Q^{*}]|\to\infty
$$
as $n\to\infty$.

Recall that $V_n(x)=\frac{1}{(B_n^{-1})'(x)}W_n(B_n^{-1}(x))$. Let
$\F_n=B_n(\F )$ and $\dot{s}_n:=\dot{t}_n\circ B_n^{-1}$. Then $V_n$
is a vector field corresponding to $\dot{s}_n$ whose support is an
ideal triangulation $\F_n$ and which is normalized to be zero at the
vertices of the complementary triangle $\Delta_n:=B_n(\Delta_0)$.
There are two possibilities to consider:

(a) $\F_n$ converges to an ideal triangulation $\F^{*}$ of $\H$, or

(b) $\F_n$ does not converge to an ideal triangulation of $\H$.

In both cases we are allowed to further normalize the sequence $V_n$
to vanish at any three points on $\hat{\R}$ by adding a quadratic
polynomial.

\vskip .2 cm

Assume we are in the case (a). Then the sequence $\dot{s}_n$ of
infinitesimal shear functions has a convergent subsequence in the
weak*-topology (if we consider $\dot{s}_n$ as functions on the space
of geodesics on $\H$). Without loss of generality, we assume that
$\dot{s}_n$ converges to $\dot{s}^{*}$ as $n\to\infty$. The support
of $\dot{s}^{*}$ is $\F^{*}$. Moreover, $\dot{s}^{*}$ satisfies
property (\ref{eq:zygmund-condition-shears}) in each fan of $\F^{*}$
because the values of $\dot{s}^{*}$ in each fan of $\F^{*}$ are
limits of the values of $\dot{s}_n$ on the fans of $\F_n$ (and all
$\dot{s}_n$ satisfy property (\ref{eq:zygmund-condition-shears})
with the same constant). Note that all shears of $\F^{*}$ are $0$
because it is the limit of $\F_n$ whose each shear is $0$. Thus
$\F^{*}$ is the image of the Farey tesselation under a M\"obius map
$B^{*}\in PSL_2(\R )$. By Remark \ref{rmk:vect_field_nonFarey}, the
infinitesimal shear function $\dot{s}^{*}$ determines a continuous
vector field $V^{*}:\R\to\R$. If the sequence $V_n$ pointwise
converges to $V^{*}$ then we get a contradiction with
$|V_n[Q^{*}]|\to\infty$ because $|V_n[Q^{*}]|\to
|V^{*}[Q^{*}]|<\infty$ as $n\to\infty$.

To obtain a contradiction in the case (a), it is enough to prove the
convergence of $V_n$ to $V^{*}$. By Remark
\ref{rmk:vect_field_nonFarey}, we have
$$
V^{*}(x)=\sum_{p\in \F^{*}_0}V_p^{*}(x)
$$
for $x\in \R$, where $\F^{*}_0$ is the set of vertices of $\F^{*}$.
The convergence is absolute and uniform on compact subsets of $\R$.
Recall that each $V_p^{*}$ is normalized to be zero on the vertices
of $\Delta_n:=B_n(\Delta_0)$. Let $\Delta^{*}$ be the limit of
$\Delta_n$. If the sequence $\Delta_n$ does not converge, then we
normalize $V_n$ to be zero on the vertices $\Delta_n'$ such that the
sequence $\Delta_n'$ converges. We denote the normalization of $V_n$
by $V_n$ again and $\Delta_n'$ by $\Delta_n$ for simplicity of
notation. Then we normalize $V^{*}(x)$ to be zero at the vertices of
the complementary triangle $\Delta^{*}$ of $\F^{*}$ by normalizing
each $V_p^{*}$ to be zero at these points.

Given a compact subset $K$ of $\R$ and an $\epsilon >0$, there
exists finitely many $p_1,p_2,\ldots ,p_k\in\F^{*}_0$ such that
\begin{equation}
\label{eq:approximation_V(x)}
|V^{*}(x)-\sum_{i=1}^kV_{p_i}^{*}(x)|<\epsilon
\end{equation}
for all $x\in K$. To see this, consider the sequence of all fans
$\F^{*}_{p_n}$ which intersect all rays $r_x$ from the center of
$\Delta^{*}$ to $x\in K$. Given $d>0$, there exists $k$ such that
the support of $V_{p_i}^{*}$ for $i>k$ has length smaller than $d$.
(This was established for the Farey tesselation $\F$ and the proof
is similar for $\F^{*}$.) The lefthand side of
(\ref{eq:approximation_V(x)}) is less than
\begin{equation}
\label{eq:sum_of_series} \sum_{i=k+1}^{\infty}|V_{p_i}^{*}(x)|\leq
C\sum_{i=k+1}^{\infty}ie^{-\frac{i\delta}{2}}
\end{equation}
where $C$ is given in the last paragraph in the proof of Proposition
\ref{prop:claim2}. The estimate (\ref{eq:approximation_V(x)})
follows for $k$ large enough.

For each fan $\F^{*}_{p_i}$, $i=1,2,\ldots ,k$, there exists a
sequence of fans $\{(\F_n)_{p_i^n}\}_{n\in\N}$ which converges to
$\F^{*}_{p_i}$ as $n\to\infty$ such that the functions
$\dot{s}_n|_{(\F_n)_{p_i^n}}$ converge in the weak* topology to
$\dot{s}_{\infty}|_{\F^{*}_{p_i}}$. We have
$$
|V_{n}(x)-\sum_{i=1}^k(V_n)_{p_i^n}(x)|<
\sum_{i=k+1}^{\infty}|(V_n)_{p_i^n}(x)|<
C_n\sum_{i=k+1}^{\infty}ie^{-\frac{i\delta}{2}}
$$
where $C_n$ is the constant corresponding to $V_n$ similar to the
above.

Note that $\{ C_n\}_{n\in\mathbf{N}}$ is a bounded sequence. Then
\begin{equation}
\label{eq:approximation_V_n(x)}
|V_n(x)-\sum_{i=1}^k(V_n)_{p_i^n}(x)|<\epsilon
\end{equation}
for all $x\in K$ and $n$ large enough. Thus $k$ can be chosen
simultaneously for $V^{*}$ and all $V_n$, $n\in\N$ depending only on
$\epsilon >0$.

By (\ref{eq:approximation_V(x)}) and
(\ref{eq:approximation_V_n(x)}), it is enough to prove that
$(V_n)_{p_i^n}(x)\to V_{p_i}^{*}(x)$ uniformly for $x$ in the
compact subset $K$ of $\R$ as $n\to\infty$. Note that
$(V_n)_{p_i^n}$ and $V_{p_i}^{*}$ are Zygmund vector fields with
uniformly bounded cross-ratio norms and the supports of
$(V_n)_{p_i^n}$ on $\R$ converge to the support of $V_{p_i}^{*}$.
Thus the uniform convergence of $(V_n)_{p_i^n}$ to $V_{p_i}^{*}$
follows from the pointwise convergence on a dense subset of the
compact subset $K$ of $\R$ because $\{ (V_n)_{p_i^n}\}$ is a normal
family. On the other hand, the weak* convergence of
$\dot{s}_n|_{(\F_n)_{p_i^n}}$ to $\dot{s}_{\infty}|_{\F^{*}_{p_i}}$
implies that $(V_n)_{p_i^n}$ converges pointwise to $V_{p_i}^{*}$,
when $(V_n)_{p_i^n}$ are normalized to vanish at the vertices of the
triangle $\Delta_n$ in $\F_n$ with $\Delta_n\to\Delta^{*}$ as
$n\to\infty$. Thus $V_n$ converges uniformly to $V^{*}$ on $K$ and
case (a) is finished.

\vskip .2 cm

It remains to consider the case (b). The sequence $\F_n$ has a
convergent subsequence in the Hausdorff topology on closed subset of
$\H$. Without loss of generality we assume that the whole sequence
$\F_n$ converges to a subset $\F^{*}$ of the space of geodesics of
$\H$. The limit $\F^{*}$ of $\F_n$ is a geodesic lamination of $\H$
because each $\F_n$ is a geodesic lamination. Moreover, $\F^{*}$ is
not empty because there is a bounded neighborhood of the imaginary
unit $i\in\H$ which has to intersect an edge of $\F_n$ (since
complementary triangles of $\F_n$ cover $\H$ and there is a positive
upper bound on the distance of points inside an ideal triangle to
its boundary). If $\F^{*}$ has a non-empty complement, then the
connected components of the complement consist of ideal triangles
because the connected components of the complement of $\F_n$ are
ideal triangles. However, if a connected component of the complement
of $\F^{*}$ is an ideal triangle then $\F^{*}$ contains all the
images of the triangle under the group of isometries of $\H$
generated by inversions in the sides of the triangle. This implies
that $\F^{*}$ is an ideal triangulation which contradicts the case
(b). Thus $\F^{*}$ is necessarily a geodesic foliation of $\H$.

Let $l$ be a leaf of the foliation $\F^{*}$ which contains $i\in\H$
and let $\Delta_n$ be a complementary triangle of $\F_n$ which
contains $i\in\H$. Two boundary sides of $\Delta_n$ converge to $l$
as $n\to\infty$. Let $p_n\in\hat{\R}$ be the common endpoint of the
above two boundary sides. Recall that
$$
V_n:=(B_n)^{*}(W_n)
$$
with
$$
|V_n[Q^{*}]|\to\infty
$$
as $n\to\infty$. By Proposition \ref{prop:claim2} and Remark
\ref{rmk:vect_field_nonFarey}, we have
$$
V_n(x)=(V_n)_{p_n}(x)+\sum_{p'\in (B_n(\hat{Q})\setminus \{
p_n\})}(V_n)_{p'}(x).
$$
Note that $V_n$ vanishes at vertices of $B_n(\Delta_0)$, where
$\Delta_0$ is the ideal triangle with vertices $0$, $1$ and
$\infty$. We normalize $V_n$ by adding a quadratic polynomial such
that it vanishes at the vertices of $\Delta_n$, where two boundary
sides of $\Delta_n$ converge to $l\in \F^{*}$. We denote by $V_n'$
the normalized vector field. Then we have
$$
V_n'(x)=(V_n)_{p_n}'(x)+\sum_{p'\in (B_n(\hat{Q})\setminus \{
p_n\})}(V_n)_{p'}'(x)
$$
where each $(V_n)_{p'}'$ is defined to vanish at the vertices of the
triangle $\Delta_n$. Note that $\dot{s}_n=\dot{t}_n\circ B_n^{-1}$
is the infinitesimal shear function of $V_n'$ and $V_n$. Since the
cross-ratio norm is invariant under the addition of a quadratic
polynomial, we have
$$
|V_n'[Q^{*}]|\to\infty
$$
as $n\to\infty$.

Since the reference triangle $\Delta_n$ converges to a single
geodesic $l$, it follows that all geodesics of $\F_n$ except the
geodesics in the fan with tip $p_n$ have the supremum of their
Euclidean sizes going to zero as $n\to\infty$. Then, similar to the
proof of Proposition \ref{prop:claim2}, we have
$$
\sum_{p'\in (B_n(\hat{Q})\setminus \{ p_n\})}|(V_n)_{p'}'(x)|\leq
C\sum_{i=1}^{\infty}(a(n)+i)e^\frac{-(a(n)+i)\delta}{2}
$$
with $a(n)\to\infty$ as $n\to\infty$. Thus $\sum_{p'\in
(B_n(\hat{Q})\setminus \{ p_n\})}|(V_n)_{p'}'(x)|$ is arbitrary
small when $n$ is large. By Proposition \ref{prop:claim1},
$(V_n)_{p_n}'(x)$ are Zygmund bounded with the cross-ratio norms
independent of $n$. We further normalize $(V_n)_{p_n}'(x)$ by adding
a quadratic polynomial $q_n(x)$ such that
$(V_n)_{p_n}''(x):=(V_n)_{p_n}'(x)+q_n(x)$ vanishes at $0$, $1$ and
$\infty$. It follows that $\{ (V_n)_{p_n}''\}_{n\in\N}$ is a normal
family. Thus there exists a subsequence of $\{
(V_n)_{p_n}''\}_{n\in\N}$ which converges uniformly. Without loss of
generality, we assume that $(V_n)_{p_n}''$ converges to a continuous
function $V_{*}''$ uniformly on compact subsets of $\R$.

Define $V_n'':=V_n'+q_n$. Then $V_n''\to V_{*}''$ as $n\to\infty$
uniformly on compact subsets of $\R$ by the above discussion. This
implies that $|V_n[Q^{*}]|=|V_n'[Q^{*}]|=|V_n''[Q^{*}]|\to
|V''_{*}[Q^{*}]|<\infty$ as $n\to\infty$. This contradicts
$|V_n[Q^{*}]|\to\infty$ as $n\to\infty$. Thus case (b) cannot occur
as well. It follows that the sequence $\{ V_n(x)\}_{n\in\mathbf{N}}$
has uniformly bounded Zygmund norms.
\end{proof}

\paragraph{\it End of the proof of Theorem
\ref{thm:zygmund-parametrization-shears}} Proposition
\ref{prop:claim2} states that $\dot{s}$ is induced by a continuous
vector field $V$. Proposition \ref{prop:Zyg_bdd_uniform_in_shears}
states that $V$ is Zygmund bounded with the Zygmund constant
depending on the constant $C$ in (\ref{eq:zygmund-condition-shears})
which finishes the proof.
\end{proof}

\vskip .2 cm

The proof of the above theorem also establishes the uniform
convergence property of the series defining the Zygmund vector field
in terms of the infinitesimal shear function $\dot{s}:\F\to\R$. The
series giving the vector field corresponding to the infinitesimal
shear function $\dot{s}:\F\to\R$ is defined by adding over all fans
vector fields corresponding to the fans which is a particular order
of the summation of the series of elementary shear vector fields for
$\dot{s}$. This is analogous to the convergence of the C\'esaro sum
of the Fourier series of continuous functions on $S^1$(see
\cite{Kat}).

Note that there exist continuous functions on the unit disk $S^1$
whose Fourier series does not converge at a single point of $S^1$.
Similarly, the series
$V(x)=\sum_{e=(a,b)\in\F}\dot{s}(e)\chi_{[a,b]}(x)
\frac{(x-a)(x-b)}{a-b}$, where $\chi_{[a,b]}$ is the characteristic
function of the interval $[a,b]$, does not necessarily converge on
$\R\setminus\hat{\Q}$ because the infinitesimal shear function
$\dot{s}$ does not induce a bounded measured lamination in the sense
of Thurston (see \cite{Th}, \cite{Sa1}). Thus we need a summation
method similar to the situation for the Fourier series which is
given in the following corollary.

\vskip .2 cm

\begin{corollary}
\label{cor:zyg_series_estimate} Let $\dot{s}:\F\to\R$ be a shear
function which satisfies (\ref{eq:zygmund-condition-shears}) in each
fan of $\F$. Then the series defining the corresponding vector field
$$
V(x)=\sum_{p\in\hat{\Q}}V_p(x)
$$
converges absolutely and uniformly on compact subsets of $\R$, where
$V_p(x)$ is the vector field corresponding to $\dot{s}|_{\F_p}$. The
term $V_p(x)$ is a piecewise quadratic polynomial except at the tip
$p$. Moreover, we have
$$
|V(x)-\sum_{i=1}^kV_{p_i}(x)|<C\sum_{i=n}^{\infty}ie^{-\frac{(i-2)\delta}{2}}
$$
where $\delta =\log (1+\sqrt{2})^2$, $C$ is a function of the
constant in (\ref{eq:zygmund-condition-shears}) and $\{ p_1,\ldots
,p_k\}$ are all Farey number of order at most $n$. In addition, the
Zygmund norms of $V(x)$ and $\sum_{i=1}^kV_{p_i}(x)$ for
$k\in\mathbf{N}$ are bounded by constant $M(C)$, where $C$ is the
constant from (\ref{eq:zygmund-condition-shears}).
\end{corollary}

\section{The Hilbert transform in shears}

Let $V:\R\to\R$ be a Zygmund vector field on $\hat{\R}$ and let
$\dot{s}:\F\to\R$ be the corresponding infinitesimal shear function.
We use the formula
$$
V(x)=\sum_{p\in\hat{Q}}V_p(x)
$$
to find the Hilbert transform of $V$ in terms of the corresponding
infinitesimal shear function $\dot{s}$, where the order of the
summation is given by increasing Farey orders of tips
$p\in\hat{\Q}$. Recall that the fan $\F_p$ with tip $p$ is
enumerated by $\{ e_n\}_{n\in\Z}$ such that the initial point of
$e_0$ is $p$ and the terminal point of $e_1$ is $p$, where $e_n$ is
oriented to the left as seen from the reference triangle $\Delta_0$.
We define the {\it tip}-$p$ {\it infinitesimal shear function}
$\dot{s}_p:\F_p=\{ e_n\}_{n\in\mathbf{N}}\to\R$ by setting
$\dot{s}_p(e_n):=\frac{1}{2}\dot{s}(e_n)$. Thus $\dot{s}_p$
satisfies (\ref{eq:zygmund-condition-shears}) with the constant
$\frac{1}{2}C$, where $C$ is the constant of $\dot{s}$. Therefore
$$
V_p(x)=\sum_{n\in\Z}\dot{s}_p(e_n)V_{(a_n,b_n)}(x)
$$
is Zygmund bounded with the cross-ratio norm bounded independently
of the fan $\F_p$, where $e_n$ has initial point $a_n$ and terminal
point $b_n$, and $V_{(a_n,b_n)}$ is the elementary shear vector
field defined in section 3.

Since $V_p$ is Zygmund bounded, it follows that the Hilbert
transform $HV_p$ is well-defined and also Zygmund bounded (see
\cite{Zyg}, \cite{Ga}, \cite{GL}). Moreover, the cross-ratio norm of
$HV_p$ is bounded in terms of the cross-ratio norm of $V_p$.

We first give a lemma which facilitates various convergence
arguments in the rest of the paper.

\begin{lemma} \label{lem:conv_of_Hilb_transform}
Let $V_n:\R\to\R$, for $n\in\mathbf{N}$, and $V:\R\to\R$ be Zygmund
bounded vector fields on $\hat{\R}$ normalized to vanish at $0$, $1$
and $\infty$. Suppose that
$$
V_n(x)\to V(x)
$$
as $n\to\infty$ uniformly on compact subsets of $\R$ and that the
sequence $\{ V_n\}_{n\in\mathbf{N}}$ has uniformly bounded
cross-ratio norms. Then
$$
H(V_n)(x)\to H(V)(x)
$$
as $n\to\infty$ uniformly on compact subsets of $\R$.
\end{lemma}

\begin{proof}
Recall that the Hilbert transform of $V(x)$ is given by
\begin{equation*}
HV(x)=-\frac{1}{\pi}\lim_{\epsilon\to
0}\Big{[}\int_{-\infty}^{x-\epsilon}\frac{x(x-1)}{\zeta (\zeta
-1)(\zeta -x)}V(\zeta )d\zeta
+\int_{x+\epsilon}^{\infty}\frac{x(x-1)}{\zeta (\zeta -1)(\zeta
-x)}V(\zeta )d\zeta\Big{]}.
\end{equation*}

Since $V, V_n$ are normalized Zygmund vector fields with uniformly
bounded cross-ratio norms, it follows that there exists $M>0$ such
that $|V(x)|,|V_n|\leq M|x|\log |x|$ as $|x|\to\infty$ (see
\cite{GL}). Consequently,
 $\int_R^{\infty}\frac{x(x-1)}{\zeta (\zeta -1)(\zeta
-x)}V(\zeta )d\zeta$, $\int_R^{\infty}\frac{x(x-1)}{\zeta (\zeta
-1)(\zeta -x)}V_n(\zeta )d\zeta$,
$\int_{-\infty}^{-R}\frac{x(x-1)}{\zeta (\zeta -1)(\zeta -x)}V(\zeta
)d\zeta$ and $\int_{-\infty}^{-R}\frac{x(x-1)}{\zeta (\zeta
-1)(\zeta -x)}V_n(\zeta )d\zeta$ are of the order $o(R^{-\alpha})$
for some $0<\alpha <1$ (see \cite{GL}). Since $V_n$ converges to $V$
uniformly on compact subsets of $\R$, it follows that (for a fixed
$\epsilon
>0$)
$$
\int_{-\infty}^{x-\epsilon}\frac{V(\zeta )-V_n(\zeta )}{\zeta (\zeta
-1)(\zeta -x)}d\zeta + \int_{x+\epsilon}^{\infty}\frac{V(\zeta
)-V_n(\zeta )}{\zeta (\zeta -1)(\zeta -x)}d\zeta\to 0
$$
as $n\to\infty$.

To finish the proof, it is enough to show that for any $\delta >0$
there exists $\epsilon =\epsilon (\delta )>0$ and $n_0=n_0(\delta
)\in\mathbf{N}$ such that
\begin{equation}
\label{eq:pv_on_epsilon} \Big{|} p.\ v.\
\int_{x-\epsilon}^{x+\epsilon}\frac{x(x-1)}{\zeta (\zeta -1)(\zeta
-x)}[V(\zeta )-V_n(\zeta )]d\zeta\Big{|}<\delta
\end{equation}
for all $n\geq n_0$.

Let $\tilde{V}$ and $\tilde{V}_n$ be the Beurling-Ahlfors extensions
to $\H$ of $V$ and $V_n$ respectively (see \cite{GS}). The
$\bar{\partial}$-derivatives $\mu$ and $\mu_n$ of $\tilde{V}$ and
$\tilde{V}_n$ are Beltrami differentials on $\H$ corresponding to
the vector fields $V$ and $V_n$(see \cite{GS}, \cite{GL}). The
vector fields $V$ and $V_n$ have uniformly (in $n$) bounded
cross-ratio norms which implies the existence of a constant $S>0$
such that $\|\mu\|_{\infty} ,\|\mu_n\|_{\infty}<S$ for all
$n\in\mathbf{N}$. Since $V_n$ converges to $V$ uniformly on compact
subsets of $\R$, it follows that $\tilde{V}_n$ converges to
$\tilde{V}$ uniformly on compact subsets of $\H\cup\R$ when
considered as a subset of $\C$. Stoke's theorem gives
\begin{equation}
\label{eq:principal-epsilon}
\begin{split}
p.\ v.\ \int_{x-\epsilon}^{x+\epsilon}\frac{x(x-1)}{\zeta (\zeta
-1)(\zeta -x)}V(\zeta )d\zeta =
 \iint_{D_{\epsilon}(x)}\frac{x(x-1)\mu (\xi )}{\xi (\xi -1)(\xi
-x)}d\zeta d\eta -\\ -\int_{C_{\epsilon}(x)}\frac{x(x-1)}{\xi (\xi
-1)(\xi -x)}\tilde{V}(\xi )d\xi \ \ \ \ \ \ \ \ \ \ \ \ \ \
\end{split}
\end{equation}
where $\mu =\bar{\partial}\tilde{V}$, $D_{\epsilon}(x)$ is the upper
half-disk with center $x$ and radius $\epsilon$, and
$C_{\epsilon}(x)$ is the circular part of the boundary of
$D_{\epsilon}(x)$. A similar equation holds for $V_n$.

By (\ref{eq:principal-epsilon}), to prove (\ref{eq:pv_on_epsilon})
it is enough to prove that
\begin{equation}
\label{eq:half-disk} \iint_{D_{\epsilon}(x)}\Big{|}\frac{x(x-1)\mu_n
(\xi )}{\xi (\xi -1)(\xi -x)}\Big{|}d\zeta d\eta\to 0
\end{equation}
as $\epsilon\to 0$ uniformly in $n\in\mathbf{N}$, and that
\begin{equation} \label{eq:half-circle}
\int_{C_{\epsilon}(x)}\frac{x(x-1)}{\xi (\xi -1)(\xi
-x)}[\tilde{V}(\xi )-\tilde{V}_n(\xi )]d\xi\to 0
\end{equation}
as $n\to\infty$ for all $\epsilon <1$. Since $\|\mu_n\|_{\infty}\leq
S$, and the expression $\frac{x(x-1)}{\xi (\xi -1)(\xi -x)}$ has a
simple pole at $x$, and the area of $D_{\epsilon}(x)$ goes to zero
as $\epsilon\to 0$, we obtain (\ref{eq:half-disk}).

To obtain (\ref{eq:half-circle}), we change the variable $\xi
=x+\epsilon e^{i\varphi }$ to obtain
\begin{equation*}
\begin{split}
\int_{C_{\epsilon}(x)}\frac{[\tilde{V}(\xi )-\tilde{V}_n(\xi )]}{\xi
(\xi -1)(\xi -x)}d\xi =\ \ \ \ \ \ \ \ \ \ \ \ \ \ \ \ \ \ \ \ \ \ \
\ \ \ \ \
\\
=\int_0^{\pi}\frac{[\tilde{V}(x+\epsilon e^{i\varphi}
)-\tilde{V}_n(x+\epsilon e^{i\varphi} )]}{(x+\epsilon
e^{i\varphi})(x-1+\epsilon e^{i\varphi}) \epsilon
e^{i\varphi}}\epsilon i e^{i\varphi}d\varphi =
\\
=\int_0^{\pi}\frac{[\tilde{V}(x+\epsilon e^{i\varphi}
)-\tilde{V}_n(x+\epsilon e^{i\varphi} )]}{(x+\epsilon
e^{i\varphi})(x-1+\epsilon e^{i\varphi})} i d\varphi
\end{split}
\end{equation*}

The last integral converges to zero because
$\tilde{V}_n\to\tilde{V}$ uniformly on compact subsets of
$\H\cup\R$.
\end{proof}

In the following theorem we decompose the Hilbert transform $HV$
into an infinite sum of the Hilbert transforms $HV_p$ for
$p\in\hat{\Q}$ analogous to Corollary \ref{cor:zyg_series_estimate}.
The main tool in the proof is Lemma
\ref{lem:conv_of_Hilb_transform}.

\begin{theorem}
\label{thm:Hilbert-on-Zygmund} Let $V:\R\to\R$ be a Zygmund vector
field on $\hat{\R}$ and let $\dot{s}:\F\to\R$ be the corresponding
infinitesimal shear function. Then the Hilbert transform of $V$ is
given by the series
$$
HV(x)=\sum_{p\in\hat{\Q}}HV_p(x)
$$
where $HV_p$ is the Hilbert transform of $V_p$ and $V_p$ is defined
as above using the shears in the fan $\F_p$. The series converges
uniformly on compact subsets of $\R$.
\end{theorem}

\begin{remark}
We point out that the uniform convergence of the series is strong
property because the Hilbert transform is given by the principal
value of an improper integral. Moreover, the infinitesimal shear
function has variable sign which further complicates the
convergence. On the other hand, it seems that the series does not
converge absolutely. One should also note that the ``naive'' series
$HV(x)=\sum_{e=(a,b)\in\F}\dot{s}(e)HV_{(a,b)}(x)$ does not
converge, where $HV_{(a,b)}(x)$ is given by
(\ref{eq:simple_shear_Hilbert3}), (\ref{eq:simple_shear_Hilbert2}),
(\ref{eq:simple_shear_Hilbert1}), or
(\ref{eq:simple_shear_Hilbert}).
\end{remark}

\begin{proof}
Let $V_n(x)=\sum_{i=1}^kV_{p_i}(x)$ such that $\{ p_1,p_2,\dots
,p_k\}$ is the set of Farey numbers of order at most $n$. Then $V_n$
converges to $V$ uniformly on compact subsets of $\R$ by Corollary
\ref{cor:zyg_series_estimate}. Proposition
\ref{prop:Zyg_bdd_uniform_in_shears} states that each $V_n$ is a
Zygmund vector field with uniformly bounded Zygmund norms. The
theorem follows from Lemma \ref{lem:conv_of_Hilb_transform}.
\end{proof}

\vskip .2 cm

We compute the Hilbert transform $HV_p$ of $V_p$ in terms of shears.

\begin{theorem}
\label{Hilbert-single-shear} Let $\dot{s}_p:\F_p\to\R$ be a shear
function that satisfies (\ref{eq:zygmund-condition-shears}) in the
fan $\F_p=\{ e_n\}_{n\in\Z}$ and let $V_p$ be the induced vector
field. Then
\begin{equation*}
HV_p(x)=\sum_{n=-\infty}^{\infty} \dot{s}_p(e_n) H(V_{e_n})(x)
\end{equation*}
where $V_{e_n}$ is the elementary shear vector field for the
geodesic $e_n$ oriented to the left as seen from the reference
triangle $\Delta_0$.
\end{theorem}

\begin{proof}
We divide the proof into two cases $p=\infty$ and $p\neq\infty$.

Assume first that $p=\infty$. We have
\begin{equation}
\label{eq:V_infty}
\begin{split}
V_{\infty}(\zeta )=\dot{s}_{\infty}(1)(\zeta
-1)+\dot{s}_{\infty}(2)(\zeta -2)+\cdots +\ \ \ \ \ \ \
\\+\dot{s}_{\infty}(n)(\zeta
-n) =\Big{(}\sum_{i=1}^n\dot{s}_{\infty}(i)\Big{)}\zeta
-\sum_{i=1}^ni\dot{s}_{\infty}(i)
\end{split}
\end{equation}
for $1\leq n\leq\zeta \leq n+1$, we have
$$
V_{\infty}(x)=0
$$
for $0\leq \zeta\leq 1$ , and we have
\begin{equation*}
\begin{split}
V_{\infty}(\zeta )=\dot{s}_{\infty}(0)\zeta
+\dot{s}_{\infty}(-1)(\zeta +1)+\cdots +\ \ \ \ \ \ \
\\+\dot{s}_{\infty}(-n)(\zeta
+n) =\Big{(}\sum_{i=0}^n\dot{s}_{\infty}(-i)\Big{)}\zeta
+\sum_{i=0}^ni\dot{s}_{\infty}(-i)
\end{split}
\end{equation*}
for $-n-1\leq \zeta\leq -n\leq 0$.

Since $V_{\infty}$ is a Zygmund map, it follows that $H(V_{\infty})$
exists and it is also a Zygmund map. Thus
$$
-\frac{1}{\pi}\int_{-n}^n\frac{x(x-1)}{\zeta (\zeta -1)(\zeta
-x)}V_{\infty}(\zeta )d\zeta\to H(V_{\infty})(x)
$$
as $n\to\infty$ for each $x\in\R$. Define
\begin{equation}
V_{\infty}^n(x)=\left\{
\begin{array}l
V_{\infty}(x),\ \ \ \ \ \ \ \ \ \ \ \ \ \ \ \ \ \ \ \ \ \ \ \ \ \ \ \ \ \ \ \ \mbox{ for } -n\leq x\leq n\\
\Big{(}\sum_{i=1}^n\dot{s}_{\infty}(i)\Big{)}x
-\sum_{i=1}^ni\dot{s}_{\infty}(i), \ \ \ \ \mbox{ for }x>n\\
\Big{(}\sum_{i=0}^n\dot{s}_{\infty}(-i)\Big{)}x
+\sum_{i=0}^ni\dot{s}_{\infty}(-i), \mbox{ for }x<-n
\end{array}
\right.
\end{equation}

It follows that
$$
-\frac{1}{\pi}\int_{-n}^n\frac{x(x-1)}{\zeta (\zeta -1)(\zeta
-x)}V_{\infty}(\zeta )d\zeta
=-\frac{1}{\pi}\int_{-n}^n\frac{x(x-1)}{\zeta (\zeta -1)(\zeta
-x)}V_{\infty}^n(\zeta )d\zeta .
$$
Since the Hilbert transform is given by an integral kernel, it
follows that
$$H(V_{\infty}^n)(x)=\sum_{i=-n}^n\dot{s}_{\infty}(i)H(V_{e_i})(x).$$

To finish the proof of the theorem in this case, it is enough to
show that $$\int_{n}^{\infty}\frac{x(x-1)}{\zeta (\zeta -1)(\zeta
-x)}V_{\infty}^n(\zeta )d\zeta\to 0$$ and
$$\int_{-\infty}^{-n}\frac{x(x-1)}{\zeta (\zeta -1)(\zeta
-x)}V_{\infty}^n(\zeta )d\zeta\to 0$$ as $n\to \infty$.

Using elementary computations, we obtain
\begin{equation}
\label{eq:remainder1}
\begin{array}l
\int_{n}^{\infty}\frac{x(x-1)}{\zeta (\zeta -1)(\zeta
-x)}V_{\infty}^n(\zeta )d\zeta=\\
x(x-1)\int_{n}^{\infty}\frac{\big{(}\sum_{i=-n}^n\dot{s}_{\infty}(i)\big{)}\zeta
-\sum_{i=-n}^ni\dot{s}_{\infty}(i)}{\zeta (\zeta -1)(\zeta
-x)}d\zeta=\\
(\sum_{i=-n}^n\dot{s}_{\infty}(i))O(\frac{1}{n})+(\sum_{i=-n}^ni\dot{s}_{\infty}(i))O(\frac{1}{n^2})
\end{array}
\end{equation}

If the tip-$\infty$ infinitesimal shear function which satisfies
(\ref{eq:zygmund-condition-shears}) in addition satisfies
\begin{equation}
\label{eq:estimate_int_Hilbert}
\sum_{i=-n}^n\dot{s}_{\infty}(i)=o(n)
\end{equation}
then
$$
\sum_{i=-n}^ni\dot{s}_{\infty}(i)=o(n^2)
$$
and the quantity (\ref{eq:remainder1}) converges to zero as
$n\to\infty$. Thus, to finish the proof that
$\int_{n}^{\infty}\frac{x(x-1)}{\zeta (\zeta -1)(\zeta
-x)}V_{\infty}^n(\zeta )d\zeta\to 0$ as $n\to\infty$ it remains to
show (\ref{eq:estimate_int_Hilbert}) which is done in Lemma
\ref{lemma:o(n)} below. The proof of
$\int_{-\infty}^{-n}\frac{x(x-1)}{\zeta (\zeta -1)(\zeta
-x)}V_{\infty}^n(\zeta )d\zeta\to 0$ as $n\to\infty$ is similar and
left to the reader. This finishes the proof for $p=\infty$.

Assume that $p=0$ and let $\F_0=\{ e_n\}_{n\in\mathbf{Z}}$ be the
fan with tip $0$. Then $e_1=(0,\infty)$ and $e_n=(0,\frac{1}{-n+1})$
for $n\in\mathbf{Z}\setminus\{ 1\}$. Fix $n\in\mathbf{N}$. We have
$$
V_0^n(\zeta )=\dot{s}(e_0)\frac{\zeta (\zeta
-1)}{1}+\dot{s}(e_{-1})\frac{\zeta (\zeta
-\frac{1}{2})}{\frac{1}{2}}+\cdots +\dot{s}(e_{-n+1})\frac{\zeta
(\zeta -\frac{1}{n})}{\frac{1}{n}}
$$
for $0\leq\zeta\leq\frac{1}{n}$. A short computation gives
$$
V_0^n(\zeta
)=\Big{[}\sum_{i=1}^ni\dot{s}(-i+1)\Big{]}\zeta^2-\Big{[}\sum_{i=1}^n\dot{s}(-i+1)\Big{]}\zeta
$$
for $0\leq\zeta\leq\frac{1}{n}$.

By the above and by Lemma (\ref{lemma:o(n)}), we get that
$$
|V_0^n(\zeta )|\leq o(n^2)\zeta^2+o(n)\zeta
$$
for $0\leq\zeta\leq\frac{1}{n}$. Then
$$
|\int_0^{1/n}\frac{V_0^n(\zeta )}{\zeta (\zeta
-1)(\zeta-x)}d\zeta|\leq o(n^2)\frac{1}{n^2}+o(n)\frac{1}{n}\to 0
$$
as $n\to\infty$. Similarly
$$
|\int^0_{-1/n}\frac{V_0^n(\zeta )}{\zeta (\zeta
-1)(\zeta-x)}d\zeta|\to 0
$$
as $n\to\infty$. This implies $H(V_0^n)\to H(V_0)$ as $n\to\infty$.

Assume that $p\neq\infty ,0$. We have that
$$
|V_p^n(\zeta )|\leq o(n^2)(\zeta -p)^2+o(n)|\zeta -p|
$$
because the vector field $V_p^n$ is the push-forward by a M\"obius
map $B:0\mapsto p$ of the vector field $V_0^n$. The argument that
$$
\int_p^{p_{n+1}}\frac{V_p^n(\zeta )}{\zeta (\zeta -1)(\zeta -x)}\to
0
$$
as $n\to\infty$ proceeds similarly as the case $p=0$ when $x=p$, and
it is even easier when $x\neq p$ because the integrand does not have
a singularity for $n$ large. Thus $H(V_p^n)\to H(V_p)$ as
$n\to\infty$.

We obtained that $H(V_p^n)\to H(V_p)$. Since
$H(V_p^n)(x)=\sum_{i=-n}^{n} \dot{s}_p(e_i) H(V_{e_i})(x)$ the
theorem follows.
\end{proof}

\begin{lemma}
\label{lemma:o(n)} Let $\dot{s}:\mathcal{F}\to\R$ be a shear
function which satisfies (\ref{eq:zygmund-condition-shears}) in a
fan $\{ e_n\}_{n\in\mathbb{Z}}$. Then
$$
\sum_{i=k}^{k+n}\dot{s}(e_i)=o(n)
$$
for each $k\in\mathbb{Z}$.
\end{lemma}

\begin{proof}
Let $a_0=\dot{s}(e_k)$ and
$a_i=a_{i-1}+\dot{s}(e_{k+i})+\dot{s}(e_{k-i})$ for $i\geq 1$. Then
the condition (\ref{eq:zygmund-condition-shears}) becomes
\begin{equation} \label{eq:sum_a_n's} |\frac{1}{n}(a_0+a_1+\cdots +a_{n-1})|\leq C.
\end{equation}
From (\ref{eq:zygmund-condition-shears}) we also obtain
\begin{equation} \label{eq:neighbor_difference}
|a_{i}-a_{i-1}|\leq 2C
\end{equation}
for all $i\geq 1$. The statement of the lemma translates to
$a_n=o(n)$.

Assume on the contrary that there exists a sequence $i_n\to\infty$
as $n\to\infty$ and a constant $k>0$ such that
$$
|a_{i_n}|\geq ki_n
$$
for all $n$. Without loss of generality, we assume that
$$
a_{i_n}\geq ki_n
$$
and seek a contradiction. By (\ref{eq:neighbor_difference}), we get
that \begin{equation} \label{eq:increase}
\begin{array}l
a_{i_n+1}\geq a_{i_n}-2C\geq ki_n-2C \\
a_{i_n+2}\geq a_{i_n}-4C\geq ki_n-4C \\
\cdots \\
a_{i_n+j}\geq a_{i_n}-2jC\geq ki_n-2jC
\end{array}
\end{equation}
for some $j\in\mathbb{Z}$. Adding the inequalities in
(\ref{eq:increase}), we obtain
$$
a_0+a_1+\cdots +a_{i_n}+a_{i_n+1}+\cdots +a_{i_n+j}\geq
ki_nj-j(j+1)C-Ci_n.
$$
If $j=[\sqrt{i_n}]$ then
$$
\frac{1}{i_n+j}(a_0+\cdots a_{i_n+j} )\geq
\frac{ki_nj}{i_n+j}-\frac{j(j+1)}{i_n+j}C-\frac{i_nC}{i_n+j}\to\infty$$
as $n\to\infty$. This is in a contradiction with
(\ref{eq:sum_a_n's}).
\end{proof}

\section{Recovering the shears}

We gave a formula for the Hilbert transform of a Zygmund vector
field in terms of the corresponding infinitesimal shear function
(see theorems \ref{thm:Hilbert-on-Zygmund} and
\ref{Hilbert-single-shear}). We describe below how to obtain the
infinitesimal shear function $\dot{s}$ corresponding to a vector
field $V:\R\to\R$.

Let $a,b,c,d\in \hat{\R}$ (given in the counter-clockwise order) be
four vertices of an ideal hyperbolic quadrilateral which is
decomposed into the union of the triangle with vertices $a,b,d$ and
the triangle with vertices $b,c,d$. Then the value of the shear
function $\dot{s}$ on the geodesic $(b,d)$ with respect to the above
two triangles is given by (see \cite{GL})
\begin{equation}
\label{eq:shear_from_vector_field}
\begin{split}
\dot{s}((b,d))=\frac{V(c)-V(b)}{c-b}+\frac{V(d)-V(a)}{d-a}-\\
\frac{V(b)-V(a)}{b-a}-\frac{V(d)-V(c)}{d-c}.
\end{split}
\end{equation}
Formula (\ref{eq:shear_from_vector_field}) is the first variation of
the cross-ratio
$$
cr(a,b,c,d)=\frac{(c-b)(d-a)}{(b-a)(d-c)}
$$
where $a,b,c,d\in\hat{\R}$ are given in the cyclic order on $\hat{\R
}$. For this definition of the cross-ratio,
$$
\log cr(a,b,c,d)
$$
is the shear on the geodesic $(b,d)$ considered as a diagonal of the
quadrilateral with vertices $a,b,c,d$.

Using the equation (\ref{eq:shear_from_vector_field}) and theorems
\ref{thm:Hilbert-on-Zygmund} and \ref{Hilbert-single-shear}, we
immediately obtain

\begin{corollary}
\label{cor:shears_of_Hilbert(V)} Let $\dot{s}:\F\to\R$ be the shear
function of a Zygmund bounded vector field on $\hat{\R}$. Then the
infinitesimal shear function
$$
H(\dot{s}):\F\to\R
$$
of the vector field $H(V)$ obtained by taking the Hilbert transform
of $V$ is given by
\begin{equation}
\label{eq:shears_of_Hilbert(V)}
\begin{array}l
H(\dot{s})((b,d))=\sum_{p\in\hat{Q}}\sum_{n\in\Z}\dot{s}_p(e_n)\Big{[}
\frac{H(V_{e_n})(c)-H(V_{e_n})(b)}{c-b}
+\frac{H(V_{e_n})(d)-H(V_{e_n})(a)}{d-a}\\
\ \ \ \ \ \ \ \ \ \ \ \ \ \ \ \ \ \ \ \ \
-\frac{H(V_{e_n})(b)-H(V_{e_n})(a)}{b-a}
-\frac{H(V_{e_n})(d)-H(V_{e_n})(c)}{d-c}\Big{]}
\end{array}
\end{equation}
where $(b,d)\in\F$ is the common boundary side of the two
complementary triangles of $\F$ with vertices $(a,b,d)$ and
$(b,c,d)$, and $e_n$ is a geodesic of the fan $\F_p$ oriented such
that $p$ is the initial point of $e_n$ for $n\geq 1$ and $p$ is the
terminal point of $e_n$ for $n\leq 0$.
\end{corollary}

The above computation of the infinitesimal shear function
$H(\dot{s})$ can be given in terms of the hyperbolic geometry of the
underlining tesselation $\F$. Namely, let $(0,\infty )$ be the
geodesic $e_n$ and let $(b,d)$ be the geodesic whose shear
$H(\dot{s})((b,d))$ we want to compute. Let $a$ and $c$ be the other
two endpoints of the ideal quadrilateral in $\mathbf{H}\setminus
(\F\setminus \{ (b,d)\})$.

Assume first that $a,b,c,d\notin\{ 0,\infty\}$ and let
$V_{0,\infty}$ be given by (\ref{eq:simple_shear_vector_field}).
Then the contribution of $H(V_{0,\infty})$ to $H(\dot{s})((b,d))$ is
$$
\frac{b}{c-b}\log |\frac{c}{b}|+\frac{d}{d-a}\log |\frac{d}{a}|-
\frac{b}{b-a}\log |\frac{b}{a}|-\frac{d}{d-c}\log |\frac{d}{c}|.
$$

We write the above expression in terms of the cross-ratios as
\begin{equation}
\label{eq:single-hilb-shear-formula}
\begin{split}
cr(0,\infty ,c,b)\log |cr(0,b,\infty ,c)|+cr(0,\infty ,a,d)\log
|cr(0,a,\infty ,d)|-\\
cr(0,\infty ,a,b)\log |cr(0,a,\infty ,b)|-cr(0,\infty ,c,d)\log
|cr(0,b,\infty ,c)|.
\end{split}
\end{equation}

Let $\delta_{a,d}$ be the hyperbolic distance between $(0,\infty )$
and $(a,d)$, and similarly for
$\delta_{a,b},\delta_{b,c},\delta_{c,d}$. Then the expression
(\ref{eq:single-hilb-shear-formula}) can be written as
\begin{equation}
\label{eq:contrib_simple_shear_hilbert}
\begin{split}
\sinh^2\Big{(}\frac{\delta_{b,c}}{2}\Big{)}\log\coth^2
\Big{(}\frac{\delta_{b,c}}{2}\Big{)}+
\cosh^2\Big{(}\frac{\delta_{a,d}}{2})
\log\coth^2 \Big{(}\frac{\delta_{a,d}}{2}\Big{)}-\\
\cosh^2\Big{(}\frac{\delta_{a,b}}{2}\Big{)}\log\coth^2
\Big{(}\frac{\delta_{a,b}}{2}\Big{)}-
\cosh^2\Big{(}\frac{\delta_{c,d}}{2}\Big{)}\log\coth^2
\Big{(}\frac{\delta_{c,d}}{2}\Big{)}
\end{split}
\end{equation}

If $a=0$ and $b,c,d\notin\{ 0,\infty\}$, then by similar reasoning
we get that the contribution is
\begin{equation}
\label{eq:contrib_simple_shear_hilbert:a=0}
\begin{split}
\sinh^2\Big{(}\frac{\delta_{b,c}}{2}\Big{)}\log\coth^2
\Big{(}\frac{\delta_{b,c}}{2}\Big{)}+ \log\coth^2
\Big{(}\frac{\delta_{b,d}}{2}\Big{)}-\\
\cosh^2\Big{(}\frac{\delta_{c,d}}{2}\Big{)}\log\coth^2
\Big{(}\frac{\delta_{c,d}}{2}\Big{)}.
\end{split}
\end{equation}

If $d=\infty$ and $a,b,c\notin\{ 0,\infty\}$, then the contribution
is
\begin{equation}
\label{eq:contrib_simple_shear_hilbert_d=infty}
\sinh^2\Big{(}\frac{\delta_{b,c}}{2}\Big{)}\log\coth^2
\Big{(}\frac{\delta_{b,c}}{2}\Big{)}-
\cosh^2\Big{(}\frac{\delta_{a,b}}{2}\Big{)}\log\coth^2
\Big{(}\frac{\delta_{a,b}}{2}\Big{)}.
\end{equation}

If $a=0$ and $d=\infty$, then the contribution is
\begin{equation}
\label{eq:contrib_simple_shear_hilbert:a=0,d=infty}
\cosh^2\Big{(}\frac{\delta_{b,c}}{2}\Big{)}\log\coth^2
\Big{(}\frac{\delta_{b,c}}{2}\Big{)}.
\end{equation}

If $b=0$ and $d=\infty$, then the contribution is
\begin{equation}
\label{eq:contrib_simple_shear_hilbert:b=0,d=infty} \log \tan^2
(\theta_{a,c})
\end{equation}
where $\theta_{a,c}$ is the angle between geodesics $(0,\infty )$
and $(a,c)$.

Note that all the contributions to $H(\dot{s})((b,d))$ are expressed
in terms of the invariants for the positions of the geodesics with
endpoints $a,b,c,d,0,\infty$ and they remain in force when the
geodesic $(0,\infty )$ is replaced by an arbitrary geodesic. Thus we
obtain

\begin{corollary}
\label{cor:shears_of_Hilbert(V)_invariant} Let $\dot{s}:\F\to\R$ be
the infinitesimal shear function of a Zygmund bounded vector field
$V$ on $\hat{\R}$. Then the infinitesimal shear function
$$
H(\dot{s}):\F\to\R
$$
of the vector field $H(V)$ obtained by taking the Hilbert transform
of $V$ is given by
\begin{equation}
\label{eq:shears_of_Hilbert(V)_invariant}
H(\dot{s})((b,d))=\sum_{p\in\hat{Q}}\sum_{n\in\Z}\dot{s}(e_n)\Delta_{b,d}(e_n)
\end{equation}
where $(b,d)\in\F$ is the common boundary side of the two
complementary triangles of $\F$ with vertices $(a,b,d)$ and
$(b,c,d)$, and $e_n$ is a geodesic of the fan $\F_p$ oriented such
that $p$ is the initial point of $e_n$ for $n\geq 1$ and $p$ is the
terminal point of $e_n$ for $n\leq 0$, and $\Delta_{b,d}(e_n)$ is
one of the expressions (\ref{eq:contrib_simple_shear_hilbert}),
(\ref{eq:contrib_simple_shear_hilbert:a=0}),
(\ref{eq:contrib_simple_shear_hilbert_d=infty})
(\ref{eq:contrib_simple_shear_hilbert:a=0,d=infty}),
(\ref{eq:contrib_simple_shear_hilbert:b=0,d=infty}) depending on the
relative positions of $(a,b,c,d)$ with respect to $e_n$.
\end{corollary}

\section{Fourier coefficients and infinitesimal shear functions}

Let $V:S^1\to\mathbf{C}$ be a Zygmund bounded vector field on the
unit circle of $S^1$. Let $\mathcal{F}$ be the Farey tesselation of
the unit disk $\mathbf{D}$ which is obtained by taking the image of
the Farey tesselation of $\H$ under the M\"obius map
$B:\H\to\mathbf{D}$ which maps $0$, $1$ and $\infty$ onto $1$, $i$
and $-1$. In this section, $V$ is normalized to vanish at $1$, $i$
and $-1$. The infinitesimal shear function
$\dot{s}:\mathcal{F}\to\mathbf{R}$ corresponding to $V$ satisfies
the property (\ref{eq:zygmund-condition-shears}) since all the
considerations are geometric.

We express Fourier coefficients of a Zygmund bounded vector field
$V$ on the unit circle $S^1$ in terms of its infinitesimal shear
function defined on the Farey tesselation $\mathcal{F}$. Parametrize
$S^1$ by associating to each point $z\in S^1$ its angle $\phi =arg
(z)$, where $\phi\in [0,2\pi ]$. Then, for $\phi_0<\phi_1$, the
elementary shear vector field is defined by
\begin{equation*}
V_{\phi_0,\phi_1}(z)=\left\{
\begin{array}l
\frac{(z-e^{i\phi_0})(z-e^{i\phi_1})}{e^{i\phi_0}-e^{i\phi_1}},\ \mbox{for}\ \phi_0<arg(z)<\phi_1\\
0,\ \ \ \ \ \ \ \ \ \mbox{otherwise}
\end{array}
\right.
\end{equation*}

Given a fan $\mathcal{F}_p$ with the tip $p\in S^1$ of the Farey
tesselation $\mathcal{F}$ in $\mathbf{D}$, denote by $V_p$ the
vector field corresponding to a infinitesimal shear function which
agrees with $\frac{1}{2}\dot{s}$ on $\mathcal{F}_p$ and which is
zero on $\mathcal{F}\setminus\mathcal{F}_p$, and $V_p$ is normalized
to be zero at $1$, $i$ and $-1$. Namely,
$$
V_p(z)=\sum_{e_n=(e^{i\phi_0^n},e^{i\phi_1^n})\in\mathcal{F}_p}\frac{1}{2}\dot{s}(e_n)V_{\phi_0^n,\phi_1^n}
(z).
$$

Corollary \ref{cor:zyg_series_estimate} implies that
$$
V(z)=\sum_{p\in\mathcal{F}^0}V_p(z)
$$
where $\mathcal{F}^0$ is the set of all vertices of $\mathcal{F}$
and the convergence is absolute and uniform on $S^1$.

An elementary integration gives the $n$-th Fourier coefficient
$\widehat{V}_{\phi_0,\phi_1}(n)$ of $V_{\phi_0,\phi_1}$ as
\begin{equation}\label{eq:fourier_simple_vect_field}
\begin{split}
\widehat{V}_{\phi_0,\phi_1}(n)=\frac{1}{2\pi
(e^{i\phi_0}-e^{i\phi_1})}\Big{[}\frac{e^{i(2-n)\phi_1}-e^{i(2-n)\phi_0}}{i(2-n)}-
(e^{i\phi_0}+e^{i\phi_1})\times \\ \times
\frac{e^{i(1-n)\phi_1}-e^{i(1-n)\phi_0}}{i(1-n)} +
e^{i(\phi_0+\phi_1)}\frac{e^{-in\phi_1}-e^{-in\phi_0}}{-in}\Big{]}
\end{split}
\end{equation}

We have the following theorem.

\begin{theorem}
\label{thm:Fourier_vect_field} Let $V:S^1\to \mathbf{C}$ be a
Zygmund vector field and let $\dot{s}:\mathcal{F}\to\R$ be the
corresponding infinitesimal shear function. Then the $n$-th Fourier
coefficient $\widehat{V}(n)$ of the vector field $V$ is given by
$$
\widehat{V}(n)=\sum_{p\in\mathcal{F}^0}\widehat{V}_p(n),
$$
where the convergence is absolute and
$$
\widehat{V}_p(n)=\sum_{e_n=(e^{i\phi_0^n},e^{i\phi_1^n})\in\mathcal{F}_p}\frac{1}{2}\dot{s}(e_n)\widehat{V}_{\phi_0^n,
\phi_1^n}(n).
$$
\end{theorem}

\begin{proof}
Since $\sum_{p\in\mathcal{F}^0}V_p(z)$ is converging uniformly on
$S^1$ to $V(z)$, the first formula of the theorem is immediate. By
Lemma \ref{lemma:o(n)} and the proof of Theorem
\ref{thm:Hilbert-on-Zygmund}, we have that $|V_p(z)|$ in a
neighborhood of $p$ is increasing to infinity slower than the
distance to $p$. This implies that the second formula of the theorem
holds.
\end{proof}

\section{The almost complex structure for Teichm\"uller spaces of finite surfaces}

Let $S$ be a finite area hyperbolic surface with $s>0$ punctures.
The hyperbolic plane $\H$ is the universal covering of $S$ such that
the covering map $\pi :\H\to S$ is a local isometry. Let $G$ be the
group of deck transformations and let $\tau$ be a locally finite
ideal triangulation of $S$. Then $\tau$ lifts to a locally finite
ideal triangulation $\tilde{\tau}$ of $\H$ which is invariant under
$G$.

The Teichm\"uller space $T(S)$ of a Riemann surface $S$ is the
homotopy class of all marked hyperbolic surfaces $f:S\to X$ up to
post-composition by hyperbolic isometries. We denote the
Teichm\"uller class of $f:S\to X$ by $[X,f]$. The Teichm\"uller
space $T(S)$ maps into the space of functions from the set of
geodesics $|\tau |$ of the triangulation $\tau$ into the real
numbers $\R$ as follows. Let $c_1,\ldots ,c_s$ be the set of
punctures of $S$. Given a marked hyperbolic surface $f:S\to X$, we
replace each curve in $| f(\tau )|$ with the geodesic of $X$
homotopic to it relative punctures. We obtain an ideal geodesic
triangulation of the marked hyperbolic surface $X$ homotopic to
$f(\tau )$. Then we assign to the marked hyperbolic surface $X$ the
function from $|\tau |$ into $\R$ which maps each edge of $\tau$ to
its shear with respect to $f(\tau )$. In this fashion, we obtain an
injective map
$$
T(S)\to\R^{|\tau |}
$$
which is a real analytic diffeomorphism onto its image (see
\cite{Th}, \cite{Pe1}). The image of $T(S)$ consists of all
functions $s:|\tau |\to\R$ such that for each puncture $c_i$,
$i=1,\ldots ,s$, we have
$$
\sum_{j=1}^{k(i)}s(e_j^{c_i})=0
$$
where $e_j^{c_i}$, $j=1,\ldots ,k(i)$, is the set of geodesics from
$\tau$ having an ideal endpoint the cusp $c_i$ such that a single
geodesic appears twice if and only if both of its endpoint are
$c_i$.

A tangent vector to $T(S)$ at point $[X,f]\in T(S)$ is also
described by a function in $\R^{|\tau |}$ as follows. Namely, a
differentiable (in $t$) path of shear functions $s_t:|\tau |\to\R$,
$t\in (-\epsilon ,\epsilon )$, such that $s_0:|\tau |\to\R$ equals
the shear function of $[X,f]$ describes a differentiable path in
$T(S)$ through the point $[X,f]$. The derivative
$$
\frac{d}{dt}s_t|_{t=0}=\dot{s}:|\tau |\to\R
$$
represents a tangent vector $v\in T_{[X,f]}T(S)$. In general, a
function $\dot{s}\in\R^{|\tau |}$ represents a tangent vector at
$T_{[X,f]}T(S)$ if and only if for each puncture $c_i$ it satisfies
$$
\sum_{j=1}^{k(i)}\dot{s}(e_j^{c_i})=0
$$
where $e_j^{c_i}$, $j=1,\ldots ,k(i)$, is the set of geodesics from
$\tau$ having an ideal endpoint the cusp $c_i$ with possible
repeating if both ends are at the same cusp as before.

Lift the triangulation $\tau$ of $S$ to a triangulation
$\tilde{\tau}$ of $\H$. We normalize the covering map such that
$\tilde{\tau}$ contains the geodesic $(0,\infty )$ and denote by
$a_{-1}<0$ the endpoint of the geodesic in the fan of $\tilde{\tau}$
with tip $\infty$ adjacent to $(0,\infty )$ on the left. We choose
the complementary triangle $\Delta_0$ with vertices $0$, $a_{-1}$
and $\infty$ to be the reference triangle for $\tilde{\tau}$ and
orient all geodesics of $\tilde{\tau}$ to the left as seen from
$\Delta_0$. Given a vertex $p\in\hat{\R}$ of $\tilde{\tau }$, let
$\{ e_n^p\}_{n\in\Z}$ be the set of edges of $\tilde{\tau}$ which
have one endpoint $p$ such that $e_n^p$ is adjacent to $e_{n+1}^p$
for all $n\in\Z$, and that $e_0^p$ has the initial point $p$. The
set $\{ e_n^p\}_{n\in\Z}=\tilde{\tau}_p$ is said to be a fan of
$\tilde{\tau}$ with tip $p$. Then, a map
$$
\tilde{s}:|\tilde{\tau} |\to\R
$$
represent a point in $T(S)$ if and only if $\tilde{s}$ is invariant
under $G$ and for each vertex $p\in\tilde{\tau}_0$ we have
$$
\sum_{i=1}^{n(p)}\tilde{s}(e_{j+i}^p)=0
$$
where $e_{j+1}^p$ is identified with $e_{j+n(p)+1}^p$ by an element
of $G$ and $j\in\mathbf{Z}$ is arbitrary. We choose $n(p)$ to be the
smallest positive number such that $e_{j+1}^p$ is identified with
$e_{j+n(p)+1}^p$ by an element of $G$. A function
$$
\dot{\tilde{s}}:|\tilde{\tau} |\to\R
$$
represent a tangent vector to a point in $T(S)$ if and only if
$\dot{\tilde{s}}$ is invariant under $G$ and for each vertex
$p\in\tilde{\tau}_0$ we have
$$
\sum_{i=1}^{n(p)}\dot{\tilde{s}}(e_{j+i}^p)=0
$$
where $e_{j+1}^p$ is identified with $e_{j+n(p)+1}^p$ by an element
of $G$, and again we assume that $n(p)$ is the smallest such number.
Note that $\dot{\tilde{s}}$ defines a tangent vector at an arbitrary
point of $T(S)$.

Our goal in this section is to compute the Hilbert transform in
terms of $\dot{\tilde{s}}:|\tilde{\tau} |\to\R$. Let $\Delta_0$ be
the reference complementary triangle to $\tilde{\tau}$ whose ideal
vertices are $0$, $a_1$ and $\infty$. Let $\{
e_n^{\infty}\}_{n\in\Z}$ be the fan of $\tilde{\tau}$ with tip
$\infty$ such that $e_0$ has initial point $\infty$ and terminal
point $0$, $e_1$ is to the left of $e_0$ and $e_n$ is adjacent to
$e_{n+1}$ for all $n\in\Z$. Let $a_n$ be the endpoint of $e_n$
different from $\infty$. Let $V_{\infty}$ be the vector field
normalized to be zero at $0$, $a_1$ and $\infty$ whose shear
function agree with $\dot{\tilde{s}}$ on $\tilde{\tau}_{\infty}$ and
equals zero on all other geodesics of $\tilde{\tau}$. The lemma
below proves that $V_{\infty}$ is a Zygmund vector field.

\begin{lemma}
\label{lem:single_shear_Zygmund} Let $n\geq 1$ be a fixed integer
and let $\dot{\tilde{s}}:\tilde{\tau}_{\infty}\to\R$ be a function
such that
$$
\dot{\tilde{s}}(e_j^{\infty})=\dot{\tilde{s}}(e_{j+n}^{\infty})
$$
and
\begin{equation}
\label{eq:sum_shears=0}
\sum_{i=j}^{j+n-1}\dot{\tilde{s}}(e_i^{\infty})=0
\end{equation}
for all $j\in\Z$. Let $\dot{\tilde{s}}_N:\tilde{\tau}_{\infty}\to\R$
be defined by
$\dot{\tilde{s}}_N(e_i^{\infty})=\dot{\tilde{s}}(e_i^{\infty})$ if
$-Nn\leq i\leq Nn-1$, and $\dot{\tilde{s}}_N(e_i^{\infty})=0$
otherwise. Let $V_{\infty}^N$ be the vector field corresponding to
$\dot{\tilde{s}}_N$ and let $V_{\infty}$ be the vector field
corresponding to $\dot{\tilde{s}}$. Then $V_{\infty}$ and
$V_{\infty}^N$ are Zygmund vector fields whose Zygmund constants are
bounded by some $C_3>0$ for all $N\in\mathbf{N}$.
\end{lemma}

\begin{proof}
Let $g(x)=x+k$, for $k>0$, be a generator of the subgroup of $G$
fixing $\infty$. Then we have $a_i+k=a_{i+n}$ for all $i\in\Z$ by
the invariance of $\tilde{\tau}$ under $G$. We normalize
$V_{\infty}$ and $V_{\infty}^N$ to be zero on the interval
$[a_{-1},a_0=0]$. By (\ref{eq:sum_shears=0}), the contribution of
$$
\dot{\tilde{s}}(e_0^{\infty})(x-a_0)+\dot{\tilde{s}}(e_1^{\infty})(x-a_1)+
\cdots +\dot{\tilde{s}}(e_{n-1}^{\infty})(x-a_{n-1})
$$
is equal to
\begin{equation}
\label{eq:vect_on_inv_part}
-\sum_{i=0}^{n-1}\dot{\tilde{s}}(e_{i}^{\infty})a_i.
\end{equation}

Moreover, for $j>0$ and $j\equiv 0$ mod $n$, we have
\begin{equation}
\label{eq:sum_inv_part_general} \sum_{i=0}^{n-1}
\dot{\tilde{s}}(e_{j+i}^{\infty})(x-a_{j+i})=
\sum_{i=0}^{n-1}\dot{\tilde{s}}(e_{i}^{\infty})(x-\frac{j}{n}k-a_{i})=
-\sum_{i=0}^{n-1}\dot{\tilde{s}}(e_i^{\infty})a_i
\end{equation}
because
$\dot{\tilde{s}}(e_{j+i}^{\infty})=\dot{\tilde{s}}(e_i^{\infty})$
and $a_{j+i}=a_i+\frac{j}{n}k$ by the invariance of
$\dot{\tilde{s}}$ and $\tilde{\tau}$ under $G$.

If $j<0$ and $j\equiv 0$ mod $n$, then, in a similar fashion, we
obtain
$$
\sum_{i=0}^{n-1}
\dot{\tilde{s}}(e_{j-1-i}^{\infty})(x-a_{j-1-i})=-\sum_{i=0}^{n-1}\dot{\tilde{s}}(e_{-i-1}^{\infty})a_{-i-1}.
$$

Define
$$C=\max_{0<i_0<n}\{\sup_{x\in [a_{i_0},a_{i_0+1}]}|\sum_{i=0}^{i_0}\dot{\tilde{s}}(e_{i}^{\infty})(x-a_{i})|+
\sup_{x\in
[a_{-i_0-1},a_{-i_0}]}|\sum_{i=1}^{i_0}\dot{\tilde{s}}(e_{-i}^{\infty})(x-a_{-i})|\}$$

The above estimates give that
\begin{equation}
\label{eq:bound_V_invariant_shear} |V_{\infty}(x)|\leq
C+\frac{|x|}{k}C_1
\end{equation}
and
\begin{equation}
\label{eq:bound_V_N_invariant_shear} |V_{\infty}^N(x)|\leq C+C_1N
\end{equation}
 for all $x\in\R$, where $C$ is defined above and $$C_1=\max\{
|\sum_{i=0}^{n-1}\dot{\tilde{s}}(e_i^{\infty})a_i|,
|\sum_{i=0}^{n-1}\dot{\tilde{s}}(e_{-i-1}^{\infty})a_{-i-1}|\}.$$

We show that $V_{\infty}(x)$ and $V_{\infty}^N(x)$ are Lipschitz
continuous with the constants bounded independently of $N$. Let
$x\in\R$ and $t>0$. By adding a linear term and by the invariance of
$V_{\infty}(x)$ under the action of $g(x)=x+k$,  we assume that
$0\leq x\leq k$.

If $t\leq k$, then an easy argument shows that
$$
|V_{\infty}^N(x+t)-V_{\infty}^N(x)|,
|V_{\infty}(x+t)-V_{\infty}(x)|\leq C_2t
$$
and
$$
|V_{\infty}^N(x)-V_{\infty}^N(x-t)|,
|V_{\infty}(x)-V_{\infty}(x-t)|\leq C_2t,
$$
where $C_2=\max\{
1,\sum_{i=0}^{n-1}|\dot{\tilde{s}}(e_i^{\infty})|\}$.

Assume that $t>k$. Then (\ref{eq:bound_V_invariant_shear}) and
(\ref{eq:bound_V_N_invariant_shear}) imply that
$$
|V_{\infty}^N(x+t)-V_{\infty}^N(x)|,|V_{\infty}(x+t)-V_{\infty}(x)|\leq
2\frac{C_1+C_2}{k}t
$$
and
$$
|V_{\infty}^N(x)-V_{\infty}^N(x-t)|,|V_{\infty}(x)-V_{\infty}(x-t)|\leq
2\frac{C_1+C_2}{k}t,
$$
where $C_1,C_2$ are as above. This implies that $V_{\infty}$ and
$V_{\infty}^N$ are Zygmund bounded with bound
$C_3=4(\frac{C_1+C_2}{k}+C_2)$.
\end{proof}

Since $V_{\infty}$ and $V_{\infty}^N$ have their Zygmund norms
bounded independently of $N$, it follows that their cross-ratio
norms are uniformly bounded as well. Let $\tilde{\tau}^0$ be the set
of vertices of $\tilde{\tau}$. This implies that $V_p$ and $V_p^N$
are Zygmund bounded with Zygmund bounds independent of $N$ for all
$p\in\tilde{\tau}^0$, where $V_p^N,V_p$ are defined analogous to
$V_{\infty}^N,V_{\infty}$.

\begin{theorem}
\label{thm:inv_shears_vect_field} Let $S$ be a finite area punctured
hyperbolic Riemann surface equipped with an ideal geodesic
triangulation $\tau$. Let $\tilde{\tau}=\{ e_n\}_{n\in\mathbf{N}}$
be the lift to the hyperbolic plane $\H$ of $\tau$ and let $\{
V_{e_n}\}_{n\in\mathbf{N}}$ be the elementary vector fields
corresponding to the geodesics $e_n$ of $\tilde{\tau}$. Let $V(x)$
be a Zygmund vector field on $\hat{\R}$ representing an
infinitesimal deformation of $S$ and let
$\dot{\tilde{s}}:\tilde{\tau}\to\R$ be the corresponding shear
function. Then the series
\begin{equation} \label{eq:sum_zygmund_invariant_shears}
\sum_{n\in \mathbf{N}}\dot{\tilde{s}}(e_n)V_{e_n}(x)
\end{equation}
converges to $V(x)$ uniformly on compact subsets of $\R$.
\end{theorem}

\begin{proof}
We follow the idea in the proof of Theorem
\ref{thm:zygmund-parametrization-shears}. Let $\{
p_i\}_{i\in\hat{\Q}}$ be the set of vertices of $\tilde{\tau}$.
There exists $\delta
>0$, such that for any two fans $\tilde{\tau}_{p_1}$ and
$\tilde{\tau}_{p_2}$ which do not have a common geodesic the minimum
distance between the geodesics of $\tilde{\tau}_{p_1}$ and
$\tilde{\tau}_{p_2}$ is at least $\delta$. This is a direct
consequence of the invariance of $\tilde{\tau}$ under a co-finite
group $G$.

Let $I\subset\R$ be a closed, bounded interval in $\R$. For $x\in
I$, let $r_x$ denote a geodesic ray that connects $i\in\H$ to $x\in
I\subset\R$. Given $\epsilon >0$, there exists a finite set of tips
$\{ p_1,\ldots ,p_s\}\subset\tilde{\tau}^0$, $s=s(\epsilon )$, such
that if a geodesic of the fan with a tip $p\in\tilde{\tau}^0$
intersects $r_x$ for some $x\in I$ then either all geodesics in
$\tilde{\tau}_p$ have their Euclidean lengths less than $\epsilon$,
or $p\in\{ p_1,\ldots ,p_s\}$ and all fans whose geodesics intersect
$r_x$ before the geodesics of $\tilde{\tau}_p$ are also in $\{
p_1,\ldots ,p_s\}$. Let $n(p_i)$ stand for the number of elements in
$\tilde{\tau}_{p_i}$ that form a fundamental set for the action by
the subgroup of $G$ that stabilizes $p_i$. There exists
$N=N(\epsilon ,s)\in\mathbf{N}$ such that all geodesics $e_j^{p_i}$
of $\tilde{\tau}_{p_i}$, $i=1,\ldots ,s$, for either $j>Nn(p_i)$ or
$j<-Nn(p_i)$ have their Euclidean sizes less that $\epsilon$.

By our construction, the sum
$$
\sum_{i=1}^sV_{p_i}^N(x)
$$
agrees with $V(x)$ on an $\epsilon$-net for $I$. Since the Zygmund
norms of each $V_p^N$ are uniformly bounded for $p\in\tilde{\tau}^0$
and $N\in\mathbf{N}$ and since the ideal triangulation has positive
minimum distance between geodesics of two disjoint fans, it follows
that the above sum converges uniformly on $I$ to $V(x)$ as $\epsilon
\to 0$ similar to proof of Theorem
\ref{thm:zygmund-parametrization-shears}. To finish the proof, it is
enough to note that each $V_p^N$ is a finite sum of elementary shear
vector fields $V_e$ for $e\in\tilde{\tau}$ with the coefficients
$\dot{\tilde{s}}(e)$.
\end{proof}

We use the above theorem to establish a formula for the Hilbert
transform in terms of infinitesimal shear function of Zygmund vector
field.

\begin{theorem}
\label{thm:Hilbert_shear_invariant} Let $V:\R\to\R$ be a Zygmund
vector field which is invariant under a co-finite group $G$ and let
$\tilde{\tau}$ be an ideal triangulation of $\H$ invariant under
$G$. Let $\dot{\tilde{s}}:\tilde{\tau}\to\R$ be the infinitesimal
shear function of $V$. Then the Hilbert transform $H(V)$ is given by
$$
H(V)(x)=\sum_{n\in\N}\dot{\tilde{s}}(e_n)H(V_{e_n})(x)
$$
where $\{e_n\}_{n\in\N}$ are the geodesics of $\tilde{\tau}$ given
in a sequence and $H(V_{e_n})$ is defined by
(\ref{eq:simple_shear_Hilbert3}).
\end{theorem}

\begin{proof}
We established that the quantity
$$
\sum_{i=1}^sV_{p_i}^N(x)
$$
in the proof of Theorem \ref{thm:inv_shears_vect_field} converges
 uniformly on compact subsets of $\R$ and that each $V_{p_i}^N$ is Zygmund
bounded with uniformly bounded Zygmund norms. Then Proposition
\ref{prop:Zyg_bdd_uniform_in_shears} implies that
$\sum_{i=1}^sV_{p_i}^N(x)$ have uniformly bounded Zygmund norms.
Then Lemma \ref{lem:conv_of_Hilb_transform} implies that
$$
H(\sum_{i=1}^sV_{p_i}^N(x))=\sum_{i=1}^sH(V_{p_i}^N)(x)\to H(V)(x)
$$
as $\epsilon\to 0$. The theorem follows because each $V_{p_i}^N$ is
a fiite linear combination of simple shear vector fields.
\end{proof}

We also obtain an expression of the infinitesimal shear function for
the Hilbert transform similarly to the case when Zygmund vector
field is not invariant.

\begin{corollary}
\label{cor:shear_of_Hilbert_invariant} Let
$\dot{\tilde{s}}:\tilde{\tau}\to\R$ be the infinitesimal shear
function of a Zygmund vector field $V$ invariant under a co-finite
group $G$, where $\tilde{\tau}$ is an ideal triangulation of $\H$
invariant under $G$. Then the infinitesimal shear function
$$
H(\dot{\tilde{s}}):\tilde{\tau}\to\R
$$
of the vector field $H(V)$ obtained by taking the Hilbert transform
of $V$ is given by
\begin{equation}
\label{eq:shears_of_Hilbert(V)_invariant}
H(\dot{\tilde{s}})((b,d))=\sum_{n\in\mathbf{N}}\dot{\tilde{s}}(e_n)\Delta_{b,d}(e_n)
\end{equation}
where $(b,d)\in\F$ is the common boundary side of the two
complementary triangles of $\F$ with vertices $(a,b,d)$ and
$(b,c,d)$, $\tilde{\tau}=\{e_n\}_{n\in\mathbf{N}}$, and
$\Delta_{b,d}(e_n)$ is one of the expressions
(\ref{eq:contrib_simple_shear_hilbert}),
(\ref{eq:contrib_simple_shear_hilbert:a=0}),
(\ref{eq:contrib_simple_shear_hilbert_d=infty}),
(\ref{eq:contrib_simple_shear_hilbert:a=0,d=infty}),
(\ref{eq:contrib_simple_shear_hilbert:b=0,d=infty}) depending on the
relative positions of $(a,b,c,d)$ with respect to $e_n$.
\end{corollary}

\section{The Weil-Petersson metric for Teichm\"uller spaces of finite surfaces
in shear coordinates}

We described above the Hilbert transform on the tangent vectors to
the Teichm\"uller space $T(S)$ of a finite area hyperbolic surface
$S$ with $s>0$ punctures. The Hilbert transform is the almost
complex structure to $T(S)$ (see \cite{NV}). The Weil-Petersson
metric $g_{WP}$ is K\"ahler \cite{Ah1}, \cite{Wo}, \cite{Wo1}. The
symplectic two form of the Weil-Petersson metric in the shear
coordinates is simply given as twice the Thurston algebraic
intersection form on the weights induced by the tangent vectors on
the train track associated to the ideal triangulation $\tau$ of $S$
(\cite{BoS}, see also \cite{Wo}, \cite{Pe2}). Therefore, we obtain
the following formula in terms of shear coordinates.

\begin{theorem}
\label{thm:WP_in_shears} Let $S$ be a finite area hyperbolic surface
with $s>0$ punctures and let $\tau$ be an ideal triangulation of
$S$. Given two infinitesimal shear functions
$$
\dot{s}_1,\dot{s}_2:\tilde{\tau}\to\R
$$
which represent two tangent vectors $v_1,v_2$ to the Teichm\"uller
space $T(S)$ of $S$ at the basepoint $[S,id]$, the Weil-Petersson
inner product is given by
$$
g_{WP}(v_1,v_2)=2i(\dot{s}_1,H(\dot{s}_2))
$$
where $H(\dot{s}_2)$ is the infinitesimal shear function of the
Hilbert transform $H(v_2)$ and $i(\cdot ,\cdot )$ is the Thurston's
algebraic intersection number between infinitesimal shear functions
defined on $\tau$.
\end{theorem}

\section{Continuous extension of the Hilbert transform and the
Weil-Petersson metric to the boundary of the Teichm\"uller space}

Let $S$ be a finite area hyperbolic surface with $s>0$ punctures and
let $\bar{T}(S)$ be the augmented Teichm\"uller space of $S$
\cite{Abik}, \cite{bers}, \cite{wolp1}. Our goal is to show that the
Hilbert transform and the Weil-Petersson metric extend by continuity
to the completion of $T(S)$. Masur \cite{mas} proved that the
Weil-Petersson hermitian metric continuously extend to the augmented
Teichm\"uller space $\bar{T}(S)$. More recently, Roger \cite{Ro}
showed that the Weil-Petersson symplectic two form on $T(S)$ extends
by continuity to the augmented Teichm\"uller space $\bar{T}(S)$
using the shear parametrization of $T(S)$ with appropriate
degeneration of the ideal triangulation when the point is in
$\bar{T}(S)\setminus T(S)$. We show that the Hilbert transform in
terms of the shears for $T(S)$ extends by continuity to the
augmented Teichm\"uller space $\bar{T}(S)$.

One description of the manifold structure on the augmented
Teichm\"uller space $\bar{T}(S)$ is given in terms of the plumbing
coordinates for a neighborhood of a point on the boundary of $T(S)$
\cite{Abik}, \cite{bers}. Another description is given by
 extending the Fenchel-Nielsen coordinates for $T(S)$
to allow the lengths of some of the closed geodesics of the pants
decomposition to be zero \cite{Abik}, \cite{bers}, \cite{wolp1}. Let
$\sigma =\{\gamma_1,\ldots ,\gamma_k\}$ be a set of mutually
non-homotopic, non-intersecting and non-trivial simple closed curves
on $S$. We denote by $\mathcal{S}(\sigma )$ the set of all (marked)
nodded Riemann surfaces $R$ obtained from $S$ by pinching curves in
$\sigma$. Then the augmented Teichm\"uller space consists of $T(S)$
and $\mathcal{S}(\sigma )$ over all $\sigma$ as above.

We use the plumbing coordinates for our purposes. Let $R_t$, $0\leq
t<1$, be a path of marked Riemann surfaces in $T(S)$ such that its
limit point $R^{*}$ as $t\to 1$ is on the boundary of $T(S)$ in
$\mathcal{S}(\sigma )$. Therefore $R^{*}$ is a nodded Riemann
surface with nodes corresponding to the curves of $\sigma$. The path
$R_t$ is described by the plumbing construction in the neighborhood
of the nodes in $R^{*}$ (see \cite{mas}). A tangent vector to
$R^{*}$ on the boundary $\bar{T}(S)\setminus T(S)$ is represented by
a (smooth) Beltrami differential $\nu$ which is supported on
$R^{*}\setminus U$, where $U$ is the union of neighborhoods of nodes
of $R^{*}$ which correspond to the curves in $\sigma$. Then the
Beltrami differential $\nu$ on $R_t\setminus U$ represents tangent
vectors at $R_t\in T(S)$ which converge to $\nu$ on $R^{*}$. We
normalize the universal coverings $\pi_t:\H\to R_t$ and $\pi:\H\to
R^{*}$ such that they map $i\in\H$ onto a fixed point $p\in
R_t\setminus U=R^{*}\setminus U$. Since the differentiable
structures on $R_t\setminus U$ and $R^{*}\setminus U$ are the same,
we can further normalize $\pi_t,\pi$ such that they map a fixed
tangent direction at $i\in\H$ onto a fixed tangent direction at $p$.

Let $\tilde{\nu}_t=(\pi_t)_{*}\nu$ and $\tilde{\nu}=(\pi )_{*}\nu$
be the pull-backs of $\nu$ by the universal coverings. The Zygmund
vector fields $V_t$ on $\hat{\R}$ corresponding to $\tilde{\nu}_t$
converge pointwise on $\hat{\R}$ to the Zygmund vector field $V$
corresponding to $\tilde{\nu}$ because $\tilde{\nu}_t$ converges
pointwise to $\tilde{\nu}$ on $\H$. The Zygmund norms of $V_t,V$ are
uniformly bounded because
$\|\tilde{\nu}_t\|_{\infty},\|\tilde{\nu}\|_{\infty}$ are uniformly
bounded.

We describe the convergence of tangent vectors as they approach to
the boundary of $T(S)$ in terms of shear parametrization. We first
recall the convergence of the path of infinitesimal shear functions
corresponding to a path of points in $T(S)$ that converges to a
boundary point $R\in\mathcal{S}(\sigma )$ described by J. Roger
\cite{Ro}. Let $\lambda =\{ h_1,h_2,\ldots ,h_n\}$ be an ideal
triangulation of the punctured surface $S$. Consider the arcs
$h_i\cap (S\setminus\sigma )$ for $i=1,2,\dots ,n$. They divide the
surface $S\setminus \sigma$ into triangles and bigons \cite{Ro}. We
group the edges of $\lambda\cap (S\setminus\sigma )$ into homotopy
classes relative to punctures of $S$ and to $\sigma$. In this
fashion an ideal triangulation $\mu =\{ g_1,g_2,\ldots ,g_n\}$ of
$S\setminus\sigma$ is obtained. Let $k_{i,j}$ be the number of arcs
of $h_j\cap (S\setminus\sigma )$ that are homotopic to $g_i$. Then

\begin{proposition} \cite{Ro}
Let $R_t\in T(S)$ for $t\in [0,1)$ be a path which converges to
$R^{*}\in\mathcal{S}(\sigma )\subset \bar{T}(S)$ as $t\to 1$. Let
$s_t:\lambda\to\mathbf{R}$ be the infinitesimal shear function for
$R_t$ and let $s^{*}:\mu\to\mathbf{R}$ be the infinitesimal shear
function for $R^{*}$. Then
\begin{equation*}
\lim_{t\to 1} \sum_{j=1}^nk_{i,j}s_t(h_j)=s^{*}(g_i).
\end{equation*}
\end{proposition}

In other words, the value $s^{*}(g_i)$ is the limit of the sum of
the values of $s_t$ on the arcs of $\lambda\cap (S\cap\sigma )$
which are homotopic to $g_i$.

We continue to work with the path $R_t\in T(S)$ for $t\in [0,1)$
which converges to $R^{*}\in\mathcal{S}(\sigma )$ as $t\to 1$. Let
$u_t$ and $u^{*}$ be tangent vectors based at $R_t$ and $R^{*}$ such
that $R_t\to R^{*}$ and $u_t\to u^{*}$ as $t\to 1$. Let
$s_t,\dot{s}_t:\lambda\to\mathbf{R}$ and
$s^{*},\dot{s}^{*}:\mu\to\mathbf{R}$ be the infinitesimal shear
functions for $R_t$, $u_t$, $R^{*}$ and $u^{*}$, respectively. Then
the sum of $s_t,\dot{s}_t$ shears for the geodesics in $\lambda$ at
each puncture of $S$ is zero. Similarly, the sum of
$s^{*},\dot{s}^{*}$ shears for the geodesics in $\mu$ at each
puncture of $S\setminus\sigma$ is zero, where $\alpha\in\sigma$ is
considered a puncture of $S\setminus \sigma$.

\begin{proposition} \label{prop:conv_shears_at_boundary} Assume that $s_t,\dot{s}_t,s^{*},\dot{s}^{*}$ are as
above. Then
$$
\lim_{t\to 1}\sum_{j=1}^nk_{i,j}\dot{s}_t(h_j)=\dot{s}^{*}(g_i).
$$
\end{proposition}

\begin{proof}
Let $R_t^k$ be the component of $R_t\setminus\sigma$ which
corresponds to the component of $S\setminus\sigma$ which contains
arcs homotopic to $g_i$. Let $p\in R_t^k\setminus U=R^{*}\setminus
U$ be a fixed point with a fixed tangent direction.
 Let $\pi_t:\mathbf{H}\to R_t$ be the
universal covering of $R_t$ chosen such that $\pi_t(i)=p$ and a
fixed tangent direction at $i\in\H$ is mapped onto the fixed tangent
direction at $p\in R_t^k$. Let $V_t$ and $V^{*}$ be the Zygmund
vector fields corresponding to the tangent vectors $u_t$ and
$u^{*}$. By the above, we have that $V_t$ converges to $V^{*}$
pointwise on $\hat{\mathbf{R}}$.

The normalization of the universal covering $\pi_t:\mathbf{H}\to
R_t$ implies that the lifts $(\pi_t)^{-1}(\sigma )$ leave every
compact subset of $\mathbf{H}$. It follows that
$(\pi_t)^{-1}(\lambda )$ converges to $\pi^{-1}(\mu )$.

Let $\tilde{g}_i$ be a lift of $g_i$ to $\mathbf{H}$ under the
covering map $\pi$. Recall that $g_i$ corresponds to a homotopy
class of arcs in $\lambda\cap (S\setminus\sigma )$. If $g_i$ does
not end at $\sigma$, then there is a unique $h_i\in\mu$ which
corresponds to it. Let $\tilde{h}_i(t)$ be the lift of $h_i$ under
$\pi_t$ such that $\tilde{h}_i(t)\to \tilde{g}_i$ as $t\to 1$. If
$g_i$ has one or two ends at $\sigma$, then there exists finitely
many geodesics $h_{i,1},\ldots ,h_{i,j}\in\lambda$ written with
multiplicity such that $h_{i,l}\cap (S\setminus\sigma )$ for
$l=1,\ldots ,j$ is homotopic to $g_i$. Let $\tilde{h}_{i,l}(t)$ be
the lift of $h_{i,l}$ under $\pi_t$ for $l=1,\ldots ,j$ such that
$\tilde{h}_{i,l}(t)$ converges to $\tilde{g}_i$ as $t\to 1$.

We can assume that $\tilde{h}_{i,1}(t),\tilde{h}_{i,2}(t)\ldots
,\tilde{h}_{i,j}(t)$ are given in order such that each
$\tilde{h}_{i,l}$ is in between $\tilde{h}_{i,l-1}$ and
$\tilde{h}_{i,l+1}$ for $l=2,\ldots , j-1$. Then $\tilde{h}_{i,l}$
and $\tilde{h}_{i,l+1}$ share a common endpoint.

Let $a,b,c,d\in\hat{\mathbf{R}}$ be the endpoints of a quadrilateral
with the diagonal $(b,d)$. Then the shear $\dot{s}$ of the diagonal
$(b,d)$ associated to a vector field $V$ on $\hat{\mathbf{R}}$ is
given by
$$
\dot{s}((b,d))=\frac{V(c)-V(b)}{c-b}+\frac{V(d)-V(a)}{d-a}
-\frac{V(d)-V(c)}{d-c}-\frac{V(b)-V(a)}{b-a}.
$$
We write $V(c,b)$ instead of $\frac{V(c)-V(b)}{c-b}$ for short and
similar for the other expressions. Note that the expression
$\dot{s}((b,d))$ is the sum of four terms with positive or negative
sign in front. In fact, the sign in front can be determined as
follows. Orient the boundary sides of the triangle $T$ with vertices
$(a,b,d)$ such that $T$ stays on the left of the boundary sides. If
the terminal endpoint of a side $t_1$ of $T$ coincides with the
initial point of a side $t_2$ of $T$ then we say that $t_1$ {\it
comes before} $t_2$, otherwise $t_1$ {\it comes after} $t_2$. The
quadruple $(a,b,c,d)$ divides $\hat{\mathbf{R}}$ into four
intervals. Each expression above corresponds to one of the four
intervals. The sign in front of the expression is positive if the
geodesic whose endpoints agree with the endpoints of the interval
comes after the diagonal $(b,d)$, and it is negative otherwise.

Each geodesic of the sequence
$\tilde{h}_{i,1}(t),\tilde{h}_{i,2}(t)\ldots ,\tilde{h}_{i,j}(t)$
intersects one lift $\tilde{\alpha}_t$ of some $\alpha\in\sigma$
under the covering map $\pi_t$. Thus $\tilde{\alpha}_t$ separates
one endpoint of each geodesic of the sequence from $i\in\mathbf{H}$
when $t$ is close enough to $1$. The set of other endpoints is
either separated from $i\in\mathbf{H}$ by another lift
$\tilde{\alpha}'$ of a geodesic $\alpha'\in\sigma$, or all endpoints
are equal. In both cases, the sequence of geodesics
$\tilde{h}_{i,1}(t),\tilde{h}_{i,2}(t)\ldots ,\tilde{h}_{i,j}(t)$
converges to $\tilde{g}_i$ as $t\to 1$. Each two consecutive
geodesics $\tilde{h}_{i,l}(t),\tilde{h}_{i,l+1}(t)$ form a
hyperbolic wedge and we make a hyperbolic triangle by adding a
geodesic to this wedge whose endpoints are the endpoints of the two
geodesic which are not common. The added geodesic is separated from
$i\in\mathbf{H}$  by either $\tilde{\alpha}$ or $\tilde{\alpha}'$.
These added geodesic are small as $t\to 1$ in the Euclidean metric
on $\H\cup\mathbf{R}$. We add a triangle $T_1$ to the side of
$\tilde{h}_{i,1}$ opposite $\tilde{h}_{i,2}$ and a triangle $T_2$ to
the side of $\tilde{h}_{i,j}$ opposite $\tilde{h}_{i,j-1}$. The
added triangles $T_1$ and $T_2$ have all their sides in
$\tilde{\lambda}$. We obtained a sequence of adjacent triangles such
that the union of each two adjacent triangles forms a quadrilateral
whose diagonal is in the sequence
$\tilde{h}_{i,1}(t),\tilde{h}_{i,2}(t)\ldots ,\tilde{h}_{i,j}(t)$.
The added geodesics are small as $t\to 1$ and each of them appears
as a boundary side of exactly two quadrilaterals which share a
triangle.

We consider the sum $\sum_{l=1}^j\dot{\tilde{s}}_t(h_{i,l}(t))$,
where $\dot{\tilde{s}}_t(\tilde{h}_{i,l})=V_t(d_l,a_l)+V_t(c_l,b_l)-
V_t(d_l,c_l)-V_t(b_l,a_l)$ as before and $(a_l,b_l,c_l,d_l)$ are the
vertices of the quadrilateral whose diagonal is
$\tilde{h}_{i,l}(t)$. Each small added geodesic $(a',b')$ is
contained in exactly two adjacent quadrilaterals. The sign of
$V(a',b')$ is determined with respect to the diagonal of the
quadrilateral. Since two quadrilaterals have different diagonals it
follows that the two terms for $V_t(a',b')$ have different signs in
the sum and they cancel out.

We continue the analysis of the terms in the above sum. Note that
$\tilde{h}_{i,1}(t)$ and $\tilde{h}_{i,j}(t)$ appear as boundary
sides of exactly one quadrilateral and that $\tilde{h}_{i,l}(t)$ for
$l=2,\ldots ,j-1$ appear as boundary sides of exactly two
quadrilaterals. Let
$\tilde{h}_{i,l}(t)=(b_l(t),d_l(t)),\tilde{h}_{i,l+1}(t)(b_{l+1}(t),d_{l+1}(t))$
be two adjacent geodesics in the above sequence which are boundary
sides of the above triangle with one small side, where either
$b_l(t)=b_{l+1}(t)$ and $d_l(t))\neq d_{l+1}(t))$, or $b_l(t)\neq
b_{l+1}(t)$ and $d_l(t))= d_{l+1}(t))$. The triangle appears as one
half of two quadrilaterals which implies that $V_t(b_l(t),d_l(t))$
and $V_t(b_{l+1}(t),d_{l+1}(t))$ have different signs in the above
sum. It is clear that
$|V_t(b_l(t),d_l(t))-V_t(b_{l+1}(t),d_{l+1}(t))|\to 0$ as $t\to 1$
because $b_l(t)$ and $b_{l+1}(t)$, and $d_l(t)$ and $d_{l+1}(t)$ are
close for $t$ close to $1$. It follows that the total sum of terms
$V_t(b_l(t),d_l(t))$ for the sequence $\{
\tilde{h}_{i,l}(t)\}_{l=1}^j$ converges to zero as $t\to 1$.

Let $a(t)$ be the third vertex of $T_1$ and let $c(t)$ be the third
vertex of $T_2$. Then $a(t)$ and $c(t)$ converge as $t\to 1$ to the
vertices of the quadrilateral which is the union of two adjacent
triangles in $\mathbf{H}\setminus\pi^{-1}(\mu )$ and which has
$\tilde{g}_i$ as its diagonal. The above considerations imply that
\begin{equation*}
\begin{split}
\lim_{t\to 0}\sum_{l=1}^j\dot{s}_t(h_{i,l})=\lim_{t\to 0}\Big{[}
V_t(c(t),b_{i,j}(t))+V_t(d_{i,1}(t),a(t))-\\
V_t(b_{i,1}(t),a(t))-V_t(d_{i,j}(t),c(t))\Big{]}= V(c,b)+\\
V(d,a)-V(b,a)-V(d,c)=\dot{s}^{*}(g_i)
\end{split}
\end{equation*}
where $g_i=(b,d)$. The proposition follows because
$\sum_{l=1}^j\dot{s}_t(h_{i,l})=\sum_{j=1}^nk_{i,j}\dot{s}_t(h_j)$
by the definition of $k_{i,j}$.
\end{proof}

Let $u_t$ for $t\in [0,1)$ be a path of tangent vectors to $T(S)$
that limits to a tangent vector $u^{*}$ at a point
$R\in\mathcal{S}(\sigma )$ as $t\to 1$. Let $V_t,V^{*}$ be the
Zygmund vector fields corresponding to the covering maps $\pi_t,\pi
:\mathbf{H}\to R_t,R^{*}$ normalized as above. This implies that
$V_t$ converges pointwise to $V^{*}$ on $\mathbf{R}$ as $t\to 1$ and
the supremum of the Zygmund norms of $V_t$ is finite. Then
$H(V_t)(x)\to H(V)(x)$ as $t\to 1$ uniformly on compact subsets of
$\mathbf{R}$ by Lemma \ref{lem:conv_of_Hilb_transform}.

Let $\dot{\tilde{s}}_t,\dot{\tilde{s}}^{*}$ be the infinitesimal
shear functions corresponding to the tangent vectors $u_t,u^{*}$.
Then by Proposition \ref{prop:conv_shears_at_boundary} we have that
$$
\lim_{t\to 0} \sum_{j=1}^l
k_{i,j}H(\dot{s}_t)(h_{i,j})=\dot{s}^{*}(g_i).
$$

This combines with the results of Roger \cite{Ro} to imply a
topological version of a theorem of Masur:

\begin{theorem}\cite{mas} The Weil-Petersson metric on $T(S)$ extends
by continuity to the closure $\bar{T}(S)$.
\end{theorem}

\end{document}